\documentclass[a4paper]{amsart}

\usepackage{amssymb}
\usepackage{pstricks, pst-plot, pst-grad}
\usepackage[unicode]{hyperref}

\hypersetup{
    colorlinks=true,           
    citecolor=blue             
}

\title[Stokes equations in an axisymmetric domain]
{A framework for approximation of the Stokes equations in an axisymmetric domain}

\author[N.~Ericsson]{Niklas Ericsson}

\address{
  Department of Engineering Science,
  University West,
  SE--461 86 Trollh\"{a}ttan,
  Sweden and
  Department of Mathematical Sciences,
  Chalmers University of Technology and University
  of Gothenburg,
  SE--412 96 Gothenburg,
  Sweden}
\email{niklas.ericsson@hv.se}



\thanks{The author wishes to express his sincere gratitude
to Mohammad Asadzadeh for continuous support throughout the work,
and to Stig Larsson for reading and commenting on the final draft.}

\renewcommand{\vec}[1]{\underline{#1}}

\newcommand{\gradvec}{\text{g\underline{ra}d} \,}
\newcommand{\gradtensor}{\text{g\underline{\underline{ra}}d} \,}
\newcommand{\gradtensorremark}{\emph{g\underline{\underline{ra}}d} \,}
\renewcommand{\div}{\text{div} \,}
\newcommand{\divremark}{\emph{div} \,}

\newtheorem{theorem}{Theorem}[section]
\newtheorem{corollary}[theorem]{Corollary}
\newtheorem{lemma}[theorem]{Lemma}
\newtheorem{remark}[theorem]{Remark}

\numberwithin{equation}{section}

\begin{document}

\begin{abstract}
We develop a framework for solving the stationary, incompressible Stokes equations in an axisymmetric domain.
By means of Fourier expansion with respect to the angular variable, the three-dimensional Stokes problem is reduced to an equivalent,
countable family of decoupled two-dimensional problems.
By using decomposition of three-dimensional Sobolev norms we derive natural variational spaces for the two-dimensional
problems, and show that the variational formulations are well-posed.
We analyze the error due to Fourier truncation and conclude that, for data that are sufficiently regular,
it suffices to solve a small number of two-dimensional problems.
\end{abstract}

\date{\today}

\maketitle                              


\section{Introduction}
To determine approximate solutions to fluid flow problems in three-dimensional geometries is a computationally demanding task.
In this paper we present a framework for efficiently solving the stationary, incompressible Stokes equations in an axisymmetric domain $\breve{\Omega}$,
which is obtained by rotating its half section $\Omega$ around the symmetry axis.

We use Fourier expansions with respect to the angular variable $\theta$, both of the solution and the data,
to reduce the three-dimensional Stokes problem to an equivalent, countable family of decoupled
two-dimensional problems (set in $\Omega$) for the Fourier coefficients.
A natural way to approximate the three-dimensional problem is then to use Fourier truncation and,
to obtain a fully discrete scheme, compute approximate solutions to a finite number of the two-dimensional problems.

This is an established technique to approximate boundary value problems that are invariant by rotation.
Early error analysis results
for second order elliptic problems can be found in~\cite{mercier_raugel_1982},
and for Poisson's equation in domains with reentrant edges in~\cite{MR1411853};
both using finite element approximation for the two-dimensional problems.
We refer to~\cite{2020arXiv200407216C} for additional references to problems described
by Laplace or wave equations, the Lam\'{e} system, Stokes or Navier-Stokes systems, and Maxwell's equations,
and to~\cite{MR1411853} for early references to algorithms and applications.

We analyze the error due to Fourier truncation and show that, for data that are sufficiently regular with respect to $\theta$,
it suffices to solve a small number of two-dimensional problems, which makes the method efficient. Also, the decoupling of the two-dimensional
problems makes it suitable for parallel implementation. A further advantage is simplification of the mesh-generation,
which is only required for the two-dimensional half section $\Omega$.

An added complexity is that the natural, variational spaces for the Fourier coefficients turn out to be weighted Sobolev spaces, where the
weight is either the distance to the symmetry axis, or its inverse. We derive these spaces by
decomposing (through a change of variables to cylindrical coordinates) the three-dimensional norms
for the relevant spaces
$L^2(\breve{\Omega})$,
$L^2_0(\breve{\Omega})$,
$\bigl(H^1(\breve{\Omega})\bigr)^3$,
$\bigl(H^1_0(\breve{\Omega})\bigr)^3$, and
$\bigl(H^{-1}(\breve{\Omega})\bigr)^3$,
into sums over all wavenumbers.
As a result, we show that the three-dimensional spaces are isometrically isomorphic to a certain subspace of the Cartesian product,
over all wavenumbers, of the two-dimensional weighted spaces.

The characterizations of the three-dimensional spaces thus obtained are in agreement with the results in~\cite{MR1693480},
where characterizations of $H^s(\breve{\Omega})$ and $\bigl(H^s(\breve{\Omega})\bigr)^3$ by Fourier coefficients
for any positive real $s$ are derived; here non-integer order spaces are treated first and the derivation for integer order spaces is then based
on Hilbert space interpolation.

As recently shown in~\cite{2020arXiv200407216C}, the more direct approach
for Sobolev spaces of integer order (based on changing to cylindrical coordinates in the three-dimensional norms)
results in equivalent norms (compared with the norms in~\cite{MR1693480}), but where the equivalence constants (unlike in~\cite{MR1693480})
are independent of the domain.
In~\cite{2020arXiv200407216C}, characterizations of $H^m(\breve{\Omega})$
by Fourier coefficients for any positive integer $m$ are facilitated
by the introduction of new differential operators
$\partial_{\zeta} = \frac{1}{\sqrt{2}} \, ( \partial_x - i \partial_y )$ and
$\partial_{\bar{\zeta}} = \frac{1}{\sqrt{2}} \, ( \partial_x + i \partial_y )$.
The results for vector spaces are then derived from a relation between scalar and vector norms for the Fourier coefficients.
In this paper, we work with the differential operators $\partial_x$ and $\partial_y$
and, in the vector case, derive the characterization by directly rewriting the $\bigl(H^1(\breve{\Omega})\bigr)^3$-norm.
We compare our results with~\cite{2020arXiv200407216C} in Appendix~\ref{app:sec:comparison}.
We also refer to the early results in~\cite{mercier_raugel_1982}, where this direct approach
was used to characterize the scalar spaces $H^m(\breve{\Omega})$, for $m=1,2,3$.

The purpose of the present work is to give a comprehensive presentation directly aimed at the Stokes problem
providing, inter alia, detailed derivations of the relevant two-dimensional spaces and norms.
Taking as starting-point, in fact, Fourier decompositions of the three-dimensional inner products
$( \cdot, \cdot )_{L^2(\breve{\Omega})}$ ,
$( \cdot, \cdot )_{(H^1(\breve{\Omega}))^3}$, and
$( \cdot, \cdot )_{(H^1_0(\breve{\Omega}))^3}$,
additionally enables us to derive a decomposition of the negative norm $\| \cdot \|_{(H^{-1}(\breve{\Omega}))^3}$,
and to highlight the relation between the three-dimensional weak formulation of the Stokes problem and the
two-dimensional weak formulations for the Fourier coefficients.

Examples of how to build on this framework by discretizing the two-dimensional problems
can be found in~\cite{MR1693480}, where spectral methods are used,
and~\cite{MR2210089}, where two families of finite elements of order 2 (one with continuous pressure corresponding to
the Taylor-Hood element and one with discontinuous pressure) are used. The case with an axisymmetric solution
(where only the Fourier coefficient of order 0 is considered, and the angular velocity component is equal to zero),
has been treated with finite elements in~\cite{MR2262757,MR1444389,MR854382}.

A paper in preparation will be devoted to design and analysis of stabilized finite elements for the two-dimensional problems.

An outline of the paper is as follows:
\begin{itemize}
\item In Sec.~\ref{sec:model_description}, we give some examples of axisymmetric domains and state the stationary, incompressible
Stokes equations.

\item In Sec.~\ref{sec:cylindrical_coordinates}, we recall some basic formulas, and state the Stokes problem in cylindrical coordinates.

\item In Sec.~\ref{sec:Fourier_expansion}, we use Fourier expansion with respect to the angular variable to reduce the three-dimensional Stokes problem
to a countable family of two-dimensional problems.

\item In Sec.~\ref{sec:variational_spaces}, we derive natural variational spaces for the Fourier coefficients by decomposing
the relevant three-dimensional norms into sums over all wavenumbers.

\item In Sec.~\ref{sec:two-dimensional_problems}, we state variational formulations of the two-dimensional problems
and show that these are well-posed.

\item In Sec.~\ref{sec:anisotropic_spaces}, we introduce two families of anisotropic spaces that we need to analyze the error due to Fourier truncation.

\item In Sec.~\ref{sec:Fourier_truncation_error}, we prove an error estimate due to Fourier truncation.
\end{itemize}

\section{Model description}
\label{sec:model_description}
We consider fluid flow in a bounded domain $\breve{\Omega}$ which is invariant by rotation around an axis.
We begin by discussing, and give a few examples of, such domains, following the notation in~\cite{MR1693480}.
We then state the stationary, incompressible Stokes equations, which we use to model the flow.

\subsection{Axisymmetric domains}
An example of an axisymmetric domain $\breve{\Omega}$ is given in Figure~\ref{fig:axisymmetricdomain1}.
The axisymmetric domain is obtained by rotating its half
section (meridian domain) $\Omega$ around the symmetry axis.
We assume that $\Omega$ is polygonal.
\begin{figure}[ht]
\begin{center}
\psset{xunit=0.95cm,yunit=0.95cm}
\begin{pspicture}(-1,-1.25)(11.5,1.5)
\psline{->}(0,0)(1.3,0)
\rput[c](1.25,-0.25){$z$}
\psline{->}(0,0)(0,1.3)
\rput[c](-0.25,1.25){$x$}
\psline{->}(0,0)(-0.7,-0.7)
\rput[c](-0.6,-1){$y$}
\psset{plotpoints=100}
\psparametricplot{90}{270}{t cos 0.5 mul 2 add t sin 1 mul 0 add}  
\pscustom[fillstyle=gradient,gradmidpoint=1,gradbegin=black!60!white,gradend=black!10!white]
{
\psparametricplot{-90}{90}{t cos 0.5 mul 5 add t sin 1 mul 0 add}  
\psline(5,1)(2,1)
\psparametricplot{90}{-90}{t cos 0.5 mul 2 add t sin 1 mul 0 add}  
\psline(2,-1)(5,-1)
}
\psline[linestyle=dashed]{->}(2,0)(1.8290,0.9397)   
\rput[l](2,0.5){$r$}
\rput[c](3.5,0){$\breve{\Omega}$}
\psline{->}(6.5,0)(7.8,0)
\rput[c](7.75,-0.25){$z$}
\psline{->}(6.5,0)(6.5,1.3)
\rput[c](6.25,1.25){$r$}
\psline(11.5,0)(11.5,1)(8.5,1)(8.5,0)
\psline[linewidth=0.25pt](8.5,0)(11.5,0)
\rput[c](10,0.5){$\Omega$}
\end{pspicture}
\caption{A right circular cylinder. The axisymmetric domain $\breve{\Omega}$ is obtained by rotating its polygonal half
section (meridian domain) $\Omega$ around the $z$-axis. \label{fig:axisymmetricdomain1}}
\end{center}
\end{figure}
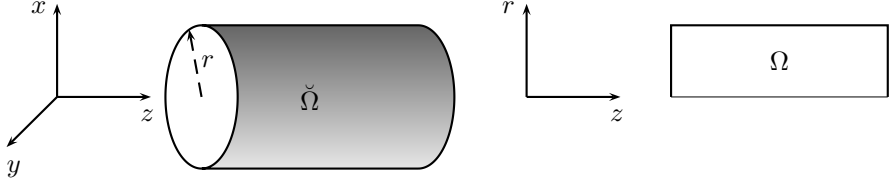
Two further examples of axisymmetric domains, and their half sections,
are given in Figure~\ref{fig:axisymmetricdomain2} and Figure~\ref{fig:axisymmetricdomain3}.
Note that the boundary $\partial \Omega = \Gamma \cup \Gamma_0$ of the half section $\Omega$ consists of two parts where
$\Gamma_0$, the interior of the part of $\partial \Omega$ contained in the symmetry axis,
is a kind of artificial boundary.
\begin{remark}
In~\cite{MR1693480}, $\Gamma_0$ is assumed to be the union of a finite number of segments with positive measure,
which means that $\Omega$ is not allowed to meet the symmetry axis at isolated points.
This assumption, as noted in~\cite{2020arXiv200407216C},
implies that $\breve{\Omega}$ (as well as its polygonal half section) is a Lipschitz domain,
and it guarantees existence of certain trace operators,
needed since~\cite{MR1693480} uses vanishing traces on $\Gamma_0$ in the
definition of the Fourier coefficient spaces.
The more direct approach for integer order Sobolev spaces used in this paper,
and more generally in~\cite{2020arXiv200407216C},
allows for more general axisymmetric domains whose intersection with the symmetry axis
is not necessarily a union of intervals and where the trace operators are not well defined.
\end{remark}
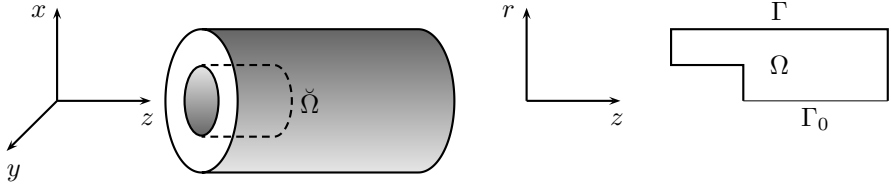
\begin{figure}[ht]
\begin{center}
\psset{xunit=0.95cm,yunit=0.95cm}
\begin{pspicture}(-1,-1.25)(11.5,1.5)
\psline{->}(0,0)(1.3,0)
\rput[c](1.25,-0.25){$z$}
\psline{->}(0,0)(0,1.3)
\rput[c](-0.25,1.25){$x$}
\psline{->}(0,0)(-0.7,-0.7)
\rput[c](-0.6,-1){$y$}
\psset{plotpoints=100}
\psparametricplot{90}{270}{t cos 0.5 mul 2 add t sin 1 mul 0 add}  
\pscustom[fillstyle=gradient,gradmidpoint=1,gradbegin=black!60!white,gradend=black!10!white]
{
\psparametricplot{-90}{90}{t cos 0.5 mul 5 add t sin 1 mul 0 add}  
\psline(5,1)(2,1)
\psparametricplot{90}{-90}{t cos 0.5 mul 2 add t sin 1 mul 0 add}  
\psline(2,-1)(5,-1)
}
\rput[c](3.5,0){$\breve{\Omega}$}
\psellipse[fillstyle=gradient,gradmidpoint=0,gradbegin=black!60!white,gradend=black!10!white](2,0)(0.25,0.5)
\pscustom[linestyle=dashed, dash=3pt 2pt]
{
\psline{-}(2,-0.5)(3,-0.5)
\psparametricplot{-90}{90}{t cos 0.25 mul 3 add t sin 0.5 mul 0 add}  
\psline{-}(3,0.5)(2,0.5)
}
\psline{->}(6.5,0)(7.8,0)
\rput[c](7.75,-0.25){$z$}
\psline{->}(6.5,0)(6.5,1.3)
\rput[c](6.25,1.25){$r$}
\psline(11.5,0)(11.5,1)(8.5,1)(8.5,0.5)(9.5,0.5)(9.5,0)
\psline[linewidth=0.25pt](9.5,0)(11.5,0)
\rput[c](10,0.5){$\Omega$}
\rput[c](10.5,-0.25){$\Gamma_0$}
\rput[c](10,1.25){$\Gamma$}
\end{pspicture}
\caption{A right circular cylinder with a hole. The boundary $\partial \Omega = \Gamma \cup \Gamma_0$
of the half section $\Omega$ consists of two parts.
$\Gamma_0$ is the interior of the part of $\partial \Omega$ contained in the $z$-axis.
Rotating the other part, $\Gamma$, around the $z$-axis gives back $\partial \breve{\Omega}$.
\label{fig:axisymmetricdomain2}}
\end{center}
\end{figure}
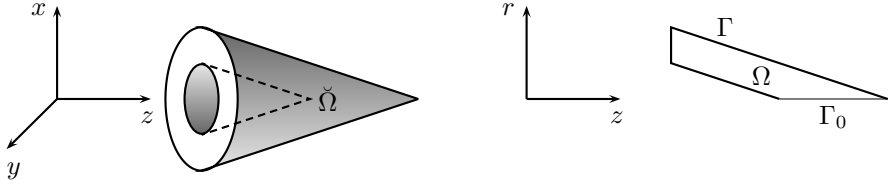
\begin{figure}[ht]
\begin{center}
\psset{xunit=0.95cm,yunit=0.95cm}
\begin{pspicture}(-1,-1.25)(11.5,1.5)
\psline{->}(0,0)(1.3,0)
\rput[c](1.25,-0.25){$z$}
\psline{->}(0,0)(0,1.3)
\rput[c](-0.25,1.25){$x$}
\psline{->}(0,0)(-0.7,-0.7)
\rput[c](-0.6,-1){$y$}
\psset{plotpoints=100}
\psparametricplot{90}{270}{t cos 0.5 mul 2 add t sin 1 mul 0 add}  
\pscustom[fillstyle=gradient,gradmidpoint=1,gradbegin=black!60!white,gradend=black!10!white]
{
\psline(5,0)(2,1)
\psparametricplot{90}{-90}{t cos 0.5 mul 2 add t sin 1 mul 0 add}  
\psline(2,-1)(5,0)
}
\rput[c](3.75,0){$\breve{\Omega}$}
\psellipse[fillstyle=gradient,gradmidpoint=0,gradbegin=black!60!white,gradend=black!10!white](2,0)(0.25,0.5)
\pscustom[linestyle=dashed, dash=3pt 2pt]
{
\psline{-}(2,-0.5)(3.5,0)(2,0.5)
}
\psline{->}(6.5,0)(7.8,0)
\rput[c](7.75,-0.25){$z$}
\psline{->}(6.5,0)(6.5,1.3)
\rput[c](6.25,1.25){$r$}
\psline(11.5,0)(8.5,1)(8.5,0.5)(10,0)
\psline[linewidth=0.25pt](10,0)(11.5,0)
\rput[c](9.75,0.325){$\Omega$}
\rput[c](10.75,-0.25){$\Gamma_0$}
\rput[c](9.25,1){$\Gamma$}
\end{pspicture}
\caption{A right circular cone with a hole. Each corner of $\Omega$ contained in the $z$-axis corresponds
to a conical singularity in $\partial \breve{\Omega}$, except if
the opening angle of $\Omega$ at this point is $\pi/2$.
Each remaining corner of $\Omega$ generates
an edge in $\partial \breve{\Omega}$.\label{fig:axisymmetricdomain3}}
\end{center}
\end{figure}

\subsection{Stokes problem}
We model fluid flow through an axisymmetric domain $\breve{\Omega}$ by
the stationary, incompressible Stokes equations
\begin{equation}
\label{eq:stokesproblem}
\left\{
  \begin{alignedat}{2}
    - \Delta \vec{\breve{u}}
      + \gradvec \breve{p} & = \vec{\breve{f}}, & \quad \; & \text{in } \breve{\Omega}, \\
      \div \vec{\breve{u}} & = 0,               &          & \text{in } \breve{\Omega}, \\
           \vec{\breve{u}} & = \vec{\breve{g}}, &          & \text{on } \partial \breve{\Omega},
  \end{alignedat}
\right.
\end{equation}
where the unknowns are the velocity $\vec{\breve{u}}$ and the pressure $\breve{p}$
\begin{align*}
\vec{\breve{u}} & = \breve{u}_x \vec{e}_x + \breve{u}_y \vec{e}_y + \breve{u}_z \vec{e}_z \in \bigl(H^1(\breve{\Omega})\bigr)^3, \\
\breve{p}       & \in L^2_0(\breve{\Omega}),
\end{align*}
and the data are the source term $\vec{\breve{f}}$ and the Dirichlet boundary data $\vec{\breve{g}}$
\begin{align*}
\vec{\breve{f}} & = \breve{f}_x \vec{e}_x + \breve{f}_y \vec{e}_y + \breve{f}_z \vec{e}_z \in \bigl(H^{-1}(\breve{\Omega})\bigr)^3, \\
\vec{\breve{g}} & = \breve{g}_x \vec{e}_x + \breve{g}_y \vec{e}_y + \breve{g}_z \vec{e}_z \in \bigl(H^{\frac{1}{2}}(\partial \breve{\Omega})\bigr)^3.
\end{align*}
For a vector field $\vec{\breve{v}} = \breve{v}_x \vec{e}_x + \breve{v}_y \vec{e}_y + \breve{v}_z \vec{e}_z$ defined on $\breve{\Omega}$,
we will let $\vec{\breve{v}}$ denote both the vector field itself and its Cartesian component vector $(\breve{v}_x, \breve{v}_y, \breve{v}_z)^T$.
Note that the divergence-free property $\div \vec{\breve{u}} = 0$ implies a necessary compatibility (zero flux) condition
on $\vec{\breve{g}}$:
\begin{equation}
\label{eq:zeroflux3d}
\int_{\partial \breve{\Omega}} \vec{\breve{g}} \cdot \vec{\breve{n}} \,\, \text{d}\breve{A} = 0,
\end{equation}
where $\vec{\breve{n}}$ denotes the unit outward normal vector to $\breve{\Omega}$ on $\partial \breve{\Omega}$,
and $\text{d}\breve{A}$ is (the magnitude of) the area element on $\partial \breve{\Omega}$.

We recall the standard definitions of the Lebesgue and Sobolev spaces
(with all derivatives being taken in the sense of distributions)
\begin{align*}
L^2(\breve{\Omega}) & := \left\{\breve{q} | \, \breve{q} \colon \breve{\Omega} \rightarrow \mathbb{C} \text{ measurable}, \;
  \int_{\breve{\Omega}} |\breve{q}|^2 \, \text{d}x \, \text{d}y \, \text{d}z < + \infty \right\}, \\
H^1(\breve{\Omega}) & := \left\{\breve{v} | \, \breve{v} \in L^{2}(\breve{\Omega}), \;
\partial_x \breve{v} \in L^{2}(\breve{\Omega}), \;
\partial_y \breve{v} \in L^{2}(\breve{\Omega}), \; \partial_z \breve{v} \in L^{2}(\breve{\Omega}) \right\},
\end{align*}
with norms
\begin{align*}
\| \breve{q} \|_{L^2(\breve{\Omega})} & := \Bigl( \int_{\breve{\Omega}} |\breve{q}|^2 \, \text{d}x \, \text{d}y \, \text{d}z \Bigr)^{\frac{1}{2}}, \\
\| \breve{v} \|_{H^1(\breve{\Omega})} & := \left( \| \breve{v} \|_{L^2(\breve{\Omega})}^2 +
                                          \| \partial_x \breve{v} \|_{L^2(\breve{\Omega})}^2 +
                                          \| \partial_y \breve{v} \|_{L^2(\breve{\Omega})}^2 +
                                          \| \partial_z \breve{v} \|_{L^2(\breve{\Omega})}^2 \right)^{\frac{1}{2}},
\end{align*}
and the corresponding inner products $( \cdot, \cdot )_{L^2(\breve{\Omega})}$
and $( \cdot, \cdot )_{H^1(\breve{\Omega})}$.
\begin{remark}
We consider spaces of complex-valued functions. In Section~\ref{sec:Fourier_expansion} we will use Fourier expansions
of the data and the unknowns, to reduce the three-dimensional Stokes
problem in $\breve{\Omega}$ to a countable family of two-dimensional problems in the half section $\Omega$.
Since the Fourier coefficients (also for real-valued functions) are complex-valued, the
two-dimensional variational spaces, that we will define in Section~\ref{sec:variational_spaces},
will be spaces of complex-valued functions.
Since the interplay between the three- and two-dimensional spaces, which will play an important
part in what follows, is developed most naturally when all spaces contain complex-valued functions,
we will use that also for the three-dimensional spaces, even though in practice the three-dimensional
data and solution will be real-valued.
\end{remark}
We also recall the subspaces
\begin{align*}
L^2_0(\breve{\Omega}) & := \left\{\breve{q} | \, \breve{q} \in L^{2}(\breve{\Omega}), \;
  \int_{\breve{\Omega}} \breve{q} \, \text{d}x \, \text{d}y \, \text{d}z = 0 \right\}, \\
H^1_0(\breve{\Omega}) & := \left\{\breve{v} | \, \breve{v} \in H^1(\breve{\Omega}), \;
\breve{v} = 0 \text{ on } \partial \breve{\Omega} \right\},
\end{align*}
and the dual of $H^1_0(\breve{\Omega})$:
\begin{align*}
H^{-1}(\breve{\Omega}) := \bigl( H^1_0(\breve{\Omega}) \bigr)^*.
\end{align*}
\begin{remark}
By the dual space $H^*$ we mean the linear space of continuous
anti-linear functionals on $H$, so that, in particular,
$(u,v)_H=\langle u,v\rangle_{H^* \times H}$ when $u,v\in H$.
\end{remark}
On $H^1_0(\breve{\Omega})$ we will use the semi-norm
\begin{align*}
| \breve{v} |_{H^1(\breve{\Omega})} & := \left( \| \partial_x \breve{v} \|_{L^2(\breve{\Omega})}^2 +
                                          \| \partial_y \breve{v} \|_{L^2(\breve{\Omega})}^2 +
                                          \| \partial_z \breve{v} \|_{L^2(\breve{\Omega})}^2 \right)^{\frac{1}{2}},
\end{align*}
which is a norm, equivalent to $\| \cdot \|_{H^1(\breve{\Omega})}$, on $H^1_0(\breve{\Omega})$,
and the corresponding inner product $( \cdot, \cdot )_{H^1_0(\breve{\Omega})}$.

Denoting by $\gamma_0$ the linear and continuous trace operator defined on $H^1(\breve{\Omega})$, we have
\begin{align*}
H^{\frac{1}{2}}(\partial \breve{\Omega}) := \gamma_0 \bigl( H^1(\breve{\Omega}) \bigr),
\end{align*}
with norm
\begin{align*}
\| \breve{g} \|_{H^{\frac{1}{2}}(\partial \breve{\Omega})} :=
\underset{\begin{array}{c} \scriptstyle \breve{v} \in H^1(\breve{\Omega}) \vspace{-1.25mm} \\
\scriptstyle \gamma_0 \breve{v} = \breve{g} \end{array}}{\inf} \| \breve{v} \|_{H^1(\breve{\Omega})}.
\end{align*}

\section{Cylindrical coordinates}
\label{sec:cylindrical_coordinates}
Since the domain $\breve{\Omega}$ is axisymmetric,
we change from Cartesian coordinates $\left(x, y, z\right)$ in
$\mathbb{R} \times \mathbb{R} \times \mathbb{R}$
to cylindrical coordinates $\left(r, \theta, z\right)$ in
$\mathbb{R}_{+} \times \left(-\pi, \pi \right] \times \mathbb{R}$, where
\begin{equation*}
\left\{
  \begin{alignedat}{1}
     r & = \sqrt{x^2+y^2}, \\
     \theta & = \left\{
     \begin{array}{rcl}
    \!\!\! -\arccos \frac{x}{r} & \text{if} & y<0, \\
            \arccos \frac{x}{r} & \text{if} & y\geq 0.
     \end{array}
     \right.
  \end{alignedat}
\right.
\end{equation*}
We define $\tilde{\Omega}$ as the product of the half section $\Omega$ and $\left(-\pi, \pi \right]$
\begin{equation}
\label{eq:omegatilde}
\tilde{\Omega} := \left\{ \left(r, \theta, z\right) | \,
\left(r, z\right) \in  \Omega, \; -\pi < \theta \leq \pi \right\},
\end{equation}
and, by analogy,
\begin{equation}
\label{eq:gammatilde}
\tilde{\Gamma} := \left\{ \left(r, \theta, z\right) | \,
\left(r, z\right) \in  \Gamma, \; -\pi < \theta \leq \pi \right\},
\end{equation}
where $\tilde{\Omega}$ and $\tilde{\Gamma}$ are point sets in cylindrical coordinates
corresponding to the domain $\breve{\Omega}$ and its boundary $\partial \breve{\Omega}$, respectively.

\subsection{Basic formulas}
We recall the identities (see Figure~\ref{fig:cylcoord} for an illustration in the $xy$-plane)
\begin{figure}[ht]
\begin{center}
\psset{xunit=0.5cm,yunit=0.5cm}
\begin{pspicture}(-0.5,-0.5)(7.5,7.5)
\psline[linewidth=0.25pt]{->}(0,0)(7,0)
\rput[c](7,-0.375){$x$}
\psline[linewidth=0.25pt]{->}(0,0)(0,7)
\rput[c](-0.375,7){$y$}
\psline{->}(0,0)(1,0)
\rput[c](1,-0.5){$\vec{e}_x$}
\psline{->}(0,0)(0,1)
\rput[c](-0.5,1){$\vec{e}_y$}
\psline[linestyle=dashed,linewidth=0.25pt](0,0)(3,1.091910703)
\rput[c](1.5,0.8){$r$}
\psarc[linewidth=0.25pt]{->}(0,0){1}{0}{20}
\rput[c](2.3,0.44){$\theta$}
\psline{->}(3,1.091910703)(3.939692621,1.433930846)
\rput[l](4.1,1.433930846){$\vec{e}_r$}
\psline{->}(3,1.091910703)(2.657979857,2.031603324)
\rput[r](2.55,2.031603324){$\vec{e}_{\theta}$}
\psline{->}(3,1.091910703)(3.694592711,5.031141715)
\rput[l](3.9,5.031141715){$\vec{\breve{v}} = \breve{v}_x \vec{e}_x + \breve{v}_y \vec{e}_y =
v_r \vec{e}_r + v_{\theta} \vec{e}_{\theta}$}
\pscircle[fillstyle=solid,fillcolor=black](3,1.091910703){0.05}
\rput[c](3,0.6){$P$}
\end{pspicture}
\caption{A point $P$ with Cartesian coordinates $\left(x, y\right)$ and
cylindrical coordinates $\left(r, \theta\right)$,
where $x= r \cos \theta$ and $y = r \sin \theta$.
A vector $\vec{\breve{v}}$, with tail in $P$, with Cartesian components
$\left(\breve{v}_x, \breve{v}_y\right)^T$
and cylindrical components $\left(v_r, v_{\theta}\right)^T$.\label{fig:cylcoord}}
\end{center}
\end{figure}
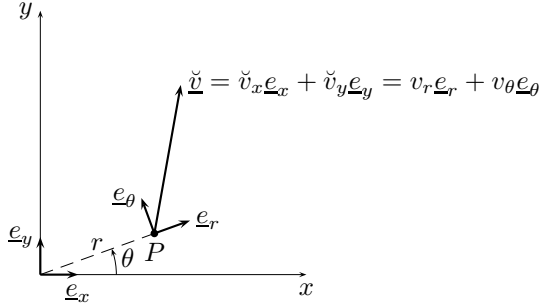
\begin{equation}
\label{eq:dxdydrdtheta}
\partial_{x} = \cos \theta \, \partial_{r} - \frac{1}{r} \, \sin \theta \, \partial_{\theta}, \quad
\partial_{y} = \sin \theta \, \partial_{r} + \frac{1}{r} \, \cos \theta \, \partial_{\theta},
\end{equation}
\begin{equation}
\label{eq:exeyeretheta}
\vec{e}_x = \cos \theta \, \vec{e}_r - \sin \theta \, \vec{e}_{\theta}, \quad
\vec{e}_y = \sin \theta \, \vec{e}_r + \cos \theta \, \vec{e}_{\theta},
\end{equation}
relating the partial derivatives and the orthonormal basis vectors
of the Cartesian and cylindrical coordinate systems, and
\begin{equation}
\label{eq:componentscyl}
\breve{v}_x = \cos \theta \, v_r - \sin \theta \, v_{\theta}, \quad
\breve{v}_y = \sin \theta \, v_r + \cos \theta \, v_{\theta}, \quad
\breve{v}_z = v_z,
\end{equation}
relating the component vectors $\vec{\breve{v}} = \left(\breve{v}_x, \breve{v}_y, \breve{v}_z\right)^T$
and $\vec{v} = \left(v_r, v_{\theta}, v_z\right)^T$
of a vector field
$\vec{\breve{v}} = \breve{v}_x \vec{e}_x + \breve{v}_y \vec{e}_y + \breve{v}_z \vec{e}_z =
v_r \vec{e}_r + v_{\theta} \vec{e}_{\theta} + v_z \vec{e}_z$ expressed in the two coordinate systems.
We will write~\eqref{eq:componentscyl} in matrix form $\vec{\breve{v}} = \mathcal{R}_{\theta} \vec{v}$, where
\begin{align*}
\mathcal{R}_{\theta} = \left(
\begin{array}{ccc}
\cos \theta & -\sin \theta & 0 \\
\sin \theta & \cos \theta & 0 \\
0 & 0 & 1
\end{array}
\right).
\end{align*}
From these identities follow the formulas for the gradient and the Laplacian operator acting on a scalar function
$\breve{v}(x, y, z) = v(r, \theta, z)$
\begin{equation}
\label{eq:gradvcyl}
\gradvec \breve{v} =
\partial_{x} \breve{v} \, \vec{e}_x + \partial_{y} \breve{v} \, \vec{e}_y + \partial_{z} \breve{v} \, \vec{e}_z =
\partial_{r} v \, \vec{e}_r + \frac{1}{r} \, \partial_{\theta} v \, \vec{e}_{\theta} + \partial_{z} v \, \vec{e}_z,
\end{equation}
\begin{equation}
\label{eq:laplacevcyl}
\begin{alignedat}{1}
\hspace{-12mm} \Delta \breve{v} =
\partial^2_x \breve{v} + \partial^2_y \breve{v} + \partial^2_z \breve{v}
& = \partial^2_r v + \frac{1}{r} \, \partial_r v + \frac{1}{r^2} \, \partial^2_{\theta} v + \partial^2_z v \\
& = \frac{1}{r} \, \partial_r ( r \partial_r v ) + \frac{1}{r^2} \, \partial^2_{\theta} v + \partial^2_z v,
\end{alignedat}
\end{equation}
and for the divergence and the vector Laplacian operator acting on a vector function
$\vec{\breve{v}}(x, y, z) = \mathcal{R}_{\theta} \vec{v}(r, \theta, z)$
\begin{equation}
\label{eq:divvcyl}
\begin{alignedat}{1}
\div \vec{\breve{v}} = \partial_x \breve{v}_x + \partial_y \breve{v}_y + \partial_z \breve{v}_z
& = \partial_r v_r + \frac{1}{r} \, v_r + \frac{1}{r} \, \partial_{\theta} v_{\theta} + \partial_z v_z \\
& = \frac{1}{r} \, \partial_r (r v_r) + \frac{1}{r} \, \partial_{\theta} v_{\theta} + \partial_z v_z,
\end{alignedat}
\end{equation}
\begin{equation}
\label{eq:veclaplacevcyl}
\begin{alignedat}{1}
\Delta \vec{\breve{v}} & = \left( \partial^2_x \breve{v}_x + \partial^2_y \breve{v}_x + \partial^2_z \breve{v}_x \right) \vec{e}_x \\
& \quad + \left( \partial^2_x \breve{v}_y + \partial^2_y \breve{v}_y + \partial^2_z \breve{v}_y \right) \vec{e}_y \\
& \qquad + \left( \partial^2_x \breve{v}_z + \partial^2_y \breve{v}_z + \partial^2_z \breve{v}_z \right) \vec{e}_z \\
& = \bigl( \frac{1}{r} \, \partial_r ( r \partial_r v_r ) + \frac{1}{r^2} \, \partial^2_{\theta} v_r + \partial^2_z v_r
    - \frac{1}{r^2} \, v_r - \frac{2}{r^2} \, \partial_{\theta} v_{\theta} \bigr) \, \vec{e}_r \\
& \quad + \bigl( \frac{1}{r} \, \partial_r ( r \partial_r v_{\theta} ) + \frac{1}{r^2} \, \partial^2_{\theta} v_{\theta}
    + \partial^2_z v_{\theta} - \frac{1}{r^2} \, v_{\theta} + \frac{2}{r^2} \, \partial_{\theta} v_r \bigr) \, \vec{e}_{\theta} \\
& \qquad + \bigl( \frac{1}{r} \, \partial_r ( r \partial_r v_z ) + \frac{1}{r^2} \, \partial^2_{\theta} v_z + \partial^2_z v_z \bigr) \, \vec{e}_z,
\end{alignedat}
\end{equation}
where the last two terms in the radial and angular components of $\Delta \vec{\breve{v}}$
result from the $\theta$-dependence of $\vec{e}_r = \cos \theta \, \vec{e}_x + \sin \theta \, \vec{e}_y$
and $\vec{e}_{\theta} = -\sin \theta \, \vec{e}_x + \cos \theta \, \vec{e}_y$,
by the relations $\partial_{\theta} \, \vec{e}_r = \vec{e}_{\theta}$ and $\partial_{\theta} \, \vec{e}_{\theta} = -\vec{e}_r$.

\subsection{Stokes problem in cylindrical coordinates}
Expressing both the data and the unknowns in cylindrical coordinates
\begin{align*}
\vec{\breve{u}} & = \breve{u}_x \vec{e}_x + \breve{u}_y \vec{e}_y + \breve{u}_z \vec{e}_z =
                    u_r \vec{e}_r + u_{\theta} \vec{e}_{\theta} + u_z \vec{e}_z, \\
\breve{p}       & = p, \\
\vec{\breve{f}} & = \breve{f}_x \vec{e}_x + \breve{f}_y \vec{e}_y + \breve{f}_z \vec{e}_z =
                    f_r \vec{e}_r + f_{\theta} \vec{e}_{\theta} + f_z \vec{e}_z, \\
\vec{\breve{g}} & = \breve{g}_x \vec{e}_x + \breve{g}_y \vec{e}_y + \breve{g}_z \vec{e}_z =
                    g_r \vec{e}_r + g_{\theta} \vec{e}_{\theta} + g_z \vec{e}_z,
\end{align*}
where $\vec{u} = \left(u_r, u_{\theta}, u_z \right)^T$, $p$, and
$\vec{f} = \left(f_r, f_{\theta}, f_z \right)^T$ are functions (distributions) on $\tilde{\Omega}$
and $\vec{g} = \left(g_r, g_{\theta}, g_z \right)^T$ on $\tilde{\Gamma}$,
from~\eqref{eq:gradvcyl}--\eqref{eq:veclaplacevcyl}
we can write the Stokes problem~\eqref{eq:stokesproblem} in cylindrical coordinates
\begin{equation}
\label{eq:stokesproblemcylcoord}
\left\{
  \begin{alignedat}{2}
    - \Delta u_r + \frac{1}{r^2} \, u_r + \frac{2}{r^2} \, \partial_{\theta} u_{\theta} \hspace{3mm}
    + \partial_{r} p & = f_r, & \quad \; & \text{in } \tilde{\Omega}, \\
    - \Delta u_{\theta} + \frac{1}{r^2} \, u_{\theta} - \frac{2}{r^2} \, \partial_{\theta} u_r
    + \frac{1}{r} \, \partial_{\theta} p & = f_{\theta}, & \quad \; & \text{in } \tilde{\Omega}, \\
    - \Delta u_z \hspace{31.5mm}
    + \partial_{z} p & = f_z, & \quad \; & \text{in } \tilde{\Omega}, \\
    \frac{1}{r} \, \partial_r (r u_r) + \frac{1}{r} \, \partial_{\theta} u_{\theta} + \partial_z u_z & = 0, & & \text{in } \tilde{\Omega}, \\
    \vec{u} & = \vec{g}, &  & \text{on } \tilde{\Gamma},
  \end{alignedat}
\right.
\end{equation}
where, from~\eqref{eq:laplacevcyl},
\begin{align*}
\Delta v = \frac{1}{r} \, \partial_r ( r \partial_r v ) + \frac{1}{r^2} \, \partial^2_{\theta} v + \partial^2_z v.
\end{align*}

\section{Fourier expansion}
\label{sec:Fourier_expansion}
A natural way to reduce the three-dimensional Stokes problem~\eqref{eq:stokesproblemcylcoord} in $\tilde{\Omega}$ to a countable family of
two-dimensional problems in $\Omega$,
is to use Fourier expansion with respect to the angular variable $\theta$, both of the solution
\begin{align}
\label{eq:fourieru}
\vec{u}(r, \theta, z) & = \frac{1}{\sqrt{2 \pi}} \, \sum_{k \in \mathbb{Z}} \vec{u}^k(r, z) e^{ik \theta}, \\
\label{eq:fourierp}
p(r, \theta, z) & = \frac{1}{\sqrt{2 \pi}} \, \sum_{k \in \mathbb{Z}} p^k(r, z) e^{ik \theta},
\end{align}
where $\vec{u}^k = \left(u_r^k, u_{\theta}^k, u_z^k \right)^T$,
and of the data
\begin{align}
\label{eq:fourierf}
\vec{f}(r, \theta, z) & = \frac{1}{\sqrt{2 \pi}} \, \sum_{k \in \mathbb{Z}} \vec{f}^k(r, z) e^{ik \theta}, \\
\label{eq:fourierg}
\vec{g}(r, \theta, z) & = \frac{1}{\sqrt{2 \pi}} \, \sum_{k \in \mathbb{Z}} \vec{g}^k(r, z) e^{ik \theta}.
\end{align}

\subsection{Two-dimensional problems}
Inserting the Fourier expansions~\eqref{eq:fourieru}--\eqref{eq:fourierg} into~\eqref{eq:stokesproblemcylcoord} results,
since the Stokes problem is linear and invariant by rotation
(which means that the coefficients of the Stokes operator in cylindrical coordinates do not depend on $\theta$),
in uncoupled two-dimensional problems for each Fourier coefficient pair $\left(\vec{u}^k, p^k \right)$,
$k \in \mathbb{Z}$:
\begin{equation}
\label{eq:stokes2dproblems}
\left\{
  \begin{alignedat}{2}
    -\Delta_a u_r^k
    + \frac{1+k^2}{r^2} \, u_r^k + \frac{2ik}{r^2} \, u_{\theta}^k \hspace{1mm}
    + \partial_{r} p^k & = f_r^k, & \quad \; & \text{in } \Omega, \\
    -\Delta_a u_{\theta}^k
    + \frac{1+k^2}{r^2} \, u_{\theta}^k - \frac{2ik}{r^2} \, u_r^k
    + \frac{ik}{r} \, p^k & = f_{\theta}^k, & \quad \; & \text{in } \Omega, \\
    -\Delta_a u_z^k
    + \frac{k^2}{r^2} \, u_z^k \hspace{21.5mm}
    + \partial_{z} p^k & = f_z^k, & \quad \; & \text{in } \Omega, \\
    \text{div}_k \, \vec{u}^k & = 0, & & \text{in } \Omega, \\
    \vec{u}^k & = \vec{g}^k, &  & \text{on } \Gamma,
  \end{alignedat}
\right.
\end{equation}
where $\Delta_a$ denotes the axisymmetric part of $\Delta$:
\begin{align*}
\Delta_a v := \frac{1}{r} \, \partial_r ( r \partial_r v ) + \partial^2_z v,
\end{align*}
and
\begin{equation}
\label{eq:divk}
\text{div}_k \, \vec{u}^k :=  \frac{1}{r} \, \partial_r (r u_r^k) + \frac{ik}{r} \, u_{\theta}^k + \partial_z u_z^k.
\end{equation}
\begin{remark}
We will, for all $k \in \mathbb{Z}$, show existence and uniqueness of solutions
to~\eqref{eq:stokes2dproblems} in Section~\ref{sec:existence_and_uniqueness}.
\end{remark}
\begin{remark}
By taking the complex conjugate of~\eqref{eq:stokes2dproblems},
it is easy to see that for real-valued data $\vec{f}$ and $\vec{g}$, in which case
(letting $\vec{\bar{f}}^{k}$, with some ambiguity of notation, denote the complex conjugate of $\vec{f}^{k}$)
\begin{align*}
\vec{f}^{-k} & = \vec{\bar{f}}^{k}, \\
\vec{g}^{-k} & = \vec{\bar{g}}^{k},
\end{align*}
the pair $\left(\vec{\bar{u}}^k, \, \bar{p}^k \right)$ solves~\eqref{eq:stokes2dproblems} with $k$ replaced by $-k$.
This means that the Fourier coefficients of the solution will also satisfy
\begin{align*}
\vec{u}^{-k} & = \vec{\bar{u}}^{k}, \\
p^{-k} & = \bar{p}^{k},
\end{align*}
(corresponding, of course, to a unique, real-valued solution of the three-dimensional Stokes problem for real-valued data)
so in the practical case with real-valued data we only need to solve the problems~\eqref{eq:stokes2dproblems} for $k \ge 0$
and, in addition, the solution for $k=0$ will be real-valued.
\end{remark}
\begin{remark}
The compatibility condition~\eqref{eq:zeroflux3d} translates into a condition
on the Fourier coefficient $\vec{g}^0$:
\begin{equation}
\label{eq:zeroflux2d}
\int_{\Gamma} \left( g_r^0 n_r + g_z^0 n_z \right) r \, \emph{d}s = 0,
\end{equation}
where $\vec{n} = \left(n_r, n_z \right)^T$ denotes the unit outward normal vector to $\Omega$ on $\Gamma$,
and $\emph{d}s$ is (the length of) the line element along $\Gamma$.
\end{remark}

\section{Variational spaces}
\label{sec:variational_spaces}
In Section~\ref{subsec:variational_formulation}, we will state variational formulations of the two-dimensional
problems~\eqref{eq:stokes2dproblems}.
To determine natural variational spaces for the Fourier coefficients defined on the half section $\Omega$,
we start by expressing the $L^2(\breve{\Omega})$- and $\bigl(H^1(\breve{\Omega})\bigr)^3$-inner products
in cylindrical coordinates, as integrals over $\tilde{\Omega}$,
and use Fourier expansions to derive decompositions of the inner products, and associated norms,
into sums over all wavenumbers.

Based on the structure of the different terms, which are weighted integrals over $\Omega$
of the Fourier coefficients and their derivatives, we define weighted Sobolev spaces on $\Omega$.

As a result, we obtain characterizations of the three-dimensional spaces
$L^2(\breve{\Omega})$ and $\bigl(H^1(\breve{\Omega})\bigr)^3$,
in terms of two-dimensional weighted spaces on $\Omega$ for all Fourier coefficients.
In particular, we show that the three-dimensional spaces are isometrically isomorphic
to a subspace of the Cartesian product, over all wavenumbers, of the two-dimensional weighted spaces.
As a corollary, we obtain corresponding characterizations of the subspaces
$L^2_0(\breve{\Omega})$ and $\bigl(H^1_0(\breve{\Omega})\bigr)^3$.

Using the results for $\bigl(H^1_0(\breve{\Omega})\bigr)^3$, we finally derive a characterization
(also in terms of spaces for all Fourier coefficients)
of its dual $\bigl(H^{-1}(\breve{\Omega})\bigr)^3$.

\subsection{Fourier decomposition of inner products and norms}
\label{subsec:Fourier_decomposition_of_norms}
In cylindrical coordinates, the $L^2(\breve{\Omega})$-inner product of two scalar functions
\begin{align}
\label{eq:q_cartesian_cyl}
\breve{p}(x, y, z) = p(r, \theta, z), \qquad
\breve{q}(x, y, z) = q(r, \theta, z),
\end{align}
is expressed as an integral over $\tilde{\Omega}$, defined by~\eqref{eq:omegatilde}, with weight $r$:
\begin{equation}
\label{eq:l2normcylcoord}
\begin{alignedat}{1}
( \breve{p}, \breve{q} )_{L^2(\breve{\Omega})}
& = \int_{\breve{\Omega}} \breve{p} \hspace{0.5pt} \bar{\breve{q}} \, \text{d}x \, \text{d}y \, \text{d}z \\
& = \int_{\tilde{\Omega}} p \hspace{0.5pt} \bar{q} \, r \, \text{d}r \, \text{d}\theta \, \text{d}z.
\end{alignedat}
\end{equation}
For two vector functions
\begin{align}
\label{eq:vec_v_cartesian_cyl}
\vec{\breve{u}}(x, y, z) = \mathcal{R}_{\theta} \vec{u}(r, \theta, z), \qquad
\vec{\breve{v}}(x, y, z) = \mathcal{R}_{\theta} \vec{v}(r, \theta, z),
\end{align}
the $\bigl(H^1(\breve{\Omega})\bigr)^3$-inner product
\begin{equation*}
\begin{alignedat}{1}
( \vec{\breve{u}}, \vec{\breve{v}} )_{(H^1(\breve{\Omega}))^3}
& = \int_{\breve{\Omega}}
  \left(
    \vec{\breve{u}} \cdot \vec{\bar{\breve{v}}} +
    \gradtensor \vec{\breve{u}} : \gradtensor \vec{\bar{\breve{v}}}
  \right) \, \text{d}x \, \text{d}y \, \text{d}z \\
& = \int_{\breve{\Omega}}
  \left(
    \breve{u}_x \bar{\breve{v}}_x + \breve{u}_y \bar{\breve{v}}_y + \breve{u}_z \bar{\breve{v}}_z \right. \\
    & \qquad \quad +        \partial_x \breve{u}_x \partial_x \bar{\breve{v}}_x
                          + \partial_y \breve{u}_x \partial_y \bar{\breve{v}}_x
                          + \partial_z \breve{u}_x \partial_z \bar{\breve{v}}_x \\
    & \qquad \quad +        \partial_x \breve{u}_y \partial_x \bar{\breve{v}}_y
                          + \partial_y \breve{u}_y \partial_y \bar{\breve{v}}_y
                          + \partial_z \breve{u}_y \partial_z \bar{\breve{v}}_y \\
    & \qquad \quad + \left. \partial_x \breve{u}_z \partial_x \bar{\breve{v}}_z
                          + \partial_y \breve{u}_z \partial_y \bar{\breve{v}}_z
                          + \partial_z \breve{u}_z \partial_z \bar{\breve{v}}_z
  \right) \, \text{d}x \, \text{d}y \, \text{d}z
\end{alignedat}
\end{equation*}
can, through
repeated use of relations~\eqref{eq:dxdydrdtheta},~\eqref{eq:componentscyl} and the Pythagorean trigonometric identity,
be expressed in cylindrical coordinates:
\begin{equation}
\label{eq:H1vecnormcylcoord}
\begin{alignedat}{1}
( \vec{\breve{u}}, \vec{\breve{v}} )_{(H^1(\breve{\Omega}))^3}
& = \int_{\tilde{\Omega}}
  \Bigl(
    u_r \bar{v}_r + u_{\theta} \bar{v}_{\theta} + u_z \bar{v}_z \\
    & \qquad \quad + \partial_r u_r \partial_r \bar{v}_r
                   + \frac{1}{r^2} \, \partial_{\theta} u_r \partial_{\theta} \bar{v}_r
                   + \partial_z u_r \partial_z \bar{v}_r
                   + \frac{1}{r^2} \, u_r \bar{v}_r \\
    & \qquad \quad + \partial_r u_{\theta} \partial_r \bar{v}_{\theta}
                   + \frac{1}{r^2} \, \partial_{\theta} u_{\theta} \partial_{\theta} \bar{v}_{\theta}
                   + \partial_z u_{\theta} \partial_z \bar{v}_{\theta}
                   + \frac{1}{r^2} \, u_{\theta} \bar{v}_{\theta} \\
    & \qquad \quad + \frac{1}{r^2} \,
    \bigl(
        (\partial_{\theta} u_{\theta}) \bar{v}_r
      + u_r (\partial_{\theta} \bar{v}_{\theta})
      - (\partial_{\theta} u_r) \bar{v}_{\theta}
      - u_{\theta} (\partial_{\theta} \bar{v}_r)
    \bigr) \\
    & \qquad \quad + \partial_r u_z \partial_r \bar{v}_z
                   + \frac{1}{r^2} \, \partial_{\theta} u_z \partial_{\theta} \bar{v}_z
                   + \partial_z u_z \partial_z \bar{v}_z
  \Bigr) \, r \, \text{d}r \, \text{d}\theta \, \text{d}z.
\end{alignedat}
\end{equation}

We now consider Fourier expansions
\begin{align}
\label{eq:fourier_q}
p & = \frac{1}{\sqrt{2 \pi}} \, \sum_{k \in \mathbb{Z}} p^k(r, z) e^{ik \theta}, \qquad
q = \frac{1}{\sqrt{2 \pi}} \, \sum_{k \in \mathbb{Z}} q^k(r, z) e^{ik \theta}, \\
\label{eq:fourier_v}
\vec{u} & = \frac{1}{\sqrt{2 \pi}} \, \sum_{k \in \mathbb{Z}} \vec{u}^k(r, z) e^{ik \theta}, \qquad
\vec{v} = \frac{1}{\sqrt{2 \pi}} \, \sum_{k \in \mathbb{Z}} \vec{v}^k(r, z) e^{ik \theta},
\end{align}
where $\vec{v} = \left(v_r, \, v_{\theta}, \, v_z \right)^T$ and
$\vec{v}^k = \left(v_r^k, \, v_{\theta}^k, \, v_z^k \right)^T$.
Using the orthogonality on $\left[-\pi, \pi \right]$ of the
family $\left\{ e^{ik \theta} \right\}_{k = -\infty}^{+\infty}$ of basis functions,
we obtain (letting $\bar{q}^k$, as before, denote the complex conjugate of $q^k$):
\begin{equation}
\label{eq:parseval1}
\begin{alignedat}{1}
\int_{-\pi}^{\pi} p \hspace{0.5pt} \bar{q} \, \text{d}\theta &  =
\frac{1}{2 \pi} \, \int_{-\pi}^{\pi}
\Bigl(
  \sum_{k \in \mathbb{Z}} p^k(r, z) e^{ik \theta}
\Bigr)
\Bigl(
  \sum_{k' \in \mathbb{Z}} \bar{q}^{k'}(r, z) e^{-ik' \theta}
\Bigr)
\, \text{d}\theta \\
& = \frac{1}{2 \pi} \,
\sum_{k, \, k' \in \mathbb{Z}}
p^k(r, z) \bar{q}^{k'}(r, z)
\int_{-\pi}^{\pi} e^{i(k-k') \theta} \, \text{d}\theta \\
& = \sum_{k \in \mathbb{Z}}
p^k(r, z) \bar{q}^{k}(r, z),
\end{alignedat}
\end{equation}
and, similiarly, in four other representative cases:
\begin{equation}
\label{eq:parseval2}
\begin{alignedat}{1}
\int_{-\pi}^{\pi} \partial_r u_r \partial_r \bar{v}_r \, \text{d}\theta  &
   = \sum_{k \in \mathbb{Z}} \partial_r u_r^k (r, z) \partial_r \bar{v}_r^k (r, z), \\
\int_{-\pi}^{\pi} \partial_{\theta} u_r \partial_{\theta} \bar{v}_r \, \text{d}\theta  &
   = \sum_{k \in \mathbb{Z}} k^2 u_r^k (r, z) \bar{v}_r^k (r, z), \\
\int_{-\pi}^{\pi} (\partial_{\theta} u_{\theta}) \bar{v}_r \, \text{d}\theta  &
= \sum_{k \in \mathbb{Z}} ik \, u_{\theta}^k(r, z) \bar{v}_r^k(r, z), \\
\int_{-\pi}^{\pi} {u}_{\theta} (\partial_{\theta} \bar{v}_r) \, \text{d}\theta  &
= \sum_{k \in \mathbb{Z}} u_{\theta}^k(r, z) (-ik) \, \bar{v}_r^k(r, z).
\end{alignedat}
\end{equation}
Inserting~\eqref{eq:parseval1}--\eqref{eq:parseval2} (and analogous results for the remaining $\theta$-integrals) into
~\eqref{eq:l2normcylcoord}, \eqref{eq:H1vecnormcylcoord} gives
\begin{equation}
\label{eq:l2innerproductcylcoord2Dsum}
( \breve{p}, \breve{q} )_{L^2(\breve{\Omega})} = \sum_{k \in \mathbb{Z}} \int_{\Omega} p^k \bar{q}^k \, r \, \text{d}r \, \text{d}z,
\end{equation}
\begin{equation}
\label{eq:H1vecinnerproductcylcoord2Dsum}
\begin{alignedat}{1}
( \vec{\breve{u}}, \vec{\breve{v}} )_{(H^1(\breve{\Omega}))^3} & =
\sum_{k \in \mathbb{Z}} \int_{\Omega}
\Bigl(
  u_r^k \bar{v}_r^k + u_{\theta}^k \bar{v}_{\theta}^k + u_z^k \bar{v}_z^k \\
  & \qquad \quad + \partial_r u_r^k \partial_r \bar{v}_r^k + \partial_z u_r^k \partial_z \bar{v}_r^k
      + \frac{1+k^2}{r^2} \, u_r^k \bar{v}_r^k + \frac{2ik}{r^2} \, u_{\theta}^k \bar{v}_r^k \\
  & \qquad \quad + \partial_r u_{\theta}^k \partial_r \bar{v}_{\theta}^k
      + \partial_z u_{\theta}^k \partial_z \bar{v}_{\theta}^k
      + \frac{1+k^2}{r^2} \, u_{\theta}^k \bar{v}_{\theta}^k
      - \frac{2ik}{r^2} \, u_r^k \bar{v}_{\theta}^k \\
  & \qquad \quad + \partial_r u_z^k \partial_r \bar{v}_z^k
  + \partial_z u_z^k \partial_z \bar{v}_z^k + \frac{k^2}{r^2} \, u_z^k \bar{v}_z^k
\Bigr) \, r \, \text{d}r \, \text{d}z,
\end{alignedat}
\end{equation}
expressing the $L^2(\breve{\Omega})$-inner product and the $\bigl(H^1(\breve{\Omega})\bigr)^3$-inner product as sums,
over all wavenumbers, of weighted integrals over the half section $\Omega$ of the Fourier coefficients
and their derivatives.

The corresponding decompositions of the associated norms are:
\begin{equation}
\label{eq:l2normcylcoord2Dsum}
\| \breve{q} \|_{L^2(\breve{\Omega})}^2 = \sum_{k \in \mathbb{Z}} \int_{\Omega} | q^k |^2 \, r \, \text{d}r \, \text{d}z,
\end{equation}
\begin{equation}
\label{eq:H1vecnormcylcoord2Dsum}
\begin{alignedat}{1}
\| \vec{\breve{v}} \|_{(H^1(\breve{\Omega}))^3}^2 & =
\sum_{k \in \mathbb{Z}} \int_{\Omega}
\Bigl(
  | v_r^k |^2 + | v_{\theta}^k |^2 + | v_z^k |^2 \\
  & \qquad \quad + | \partial_r v_r^k |^2 + | \partial_z v_r^k |^2
      + \frac{1+k^2}{r^2} \, | v_r^k |^2 + \frac{2ik}{r^2} \, v_{\theta}^k \bar{v}_r^k \\
  & \qquad \quad + | \partial_r v_{\theta}^k |^2 + | \partial_z v_{\theta}^k |^2
      + \frac{1+k^2}{r^2} \, | v_{\theta}^k |^2 - \frac{2ik}{r^2} \, v_r^k \bar{v}_{\theta}^k \\
  & \qquad \quad + | \partial_r v_z^k |^2 + | \partial_z v_z^k |^2 + \frac{k^2}{r^2} \, | v_z^k |^2
\Bigr) \, r \, \text{d}r \, \text{d}z.
\end{alignedat}
\end{equation}

\subsection{Weighted Sobolev spaces on \texorpdfstring{$\Omega$}{\unichar{"03A9}}}
\label{subsec:weighted_Sobolev_spaces}
Led by~\eqref{eq:l2normcylcoord2Dsum}--\eqref{eq:H1vecnormcylcoord2Dsum},
where each term is an integral over $\Omega$ with weight $r$ or $r^{-1}$,
we first introduce the spaces
\begin{align*}
L^2_1(\Omega) & := \left\{v | \, v \colon \Omega \rightarrow \mathbb{C} \text{ measurable}, \;
  \int_{\Omega} |v(r, z)|^2 \, r \, \text{d}r \, \text{d}z < + \infty \right\}, \\
L^2_{-1}(\Omega) & := \left\{v | \, v \colon \Omega \rightarrow \mathbb{C} \text{ measurable}, \;
  \int_{\Omega} |v(r, z)|^2 \, r^{-1} \, \text{d}r \, \text{d}z < + \infty \right\},
\end{align*}
equipped with the natural norms
\begin{align*}
\| v \|_{L^2_1(\Omega)} & := \Bigl( \int_{\Omega} |v(r, z)|^2 \, r \, \text{d}r \, \text{d}z \Bigr)^{\frac{1}{2}}, \\
\| v \|_{L^2_{-1}(\Omega)} & := \Bigl( \int_{\Omega} |v(r, z)|^2 \, r^{-1} \, \text{d}r \, \text{d}z \Bigr)^{\frac{1}{2}}.
\end{align*}

Next, we define $H^1_1(\Omega)$ as the space of functions in $L^2_1(\Omega)$
such that their partial derivatives (being taken in the sense of distributions) of order $1$
belong to $L^2_1(\Omega)$,
equipped with the semi-norm
\begin{align*}
| v |_{H^1_1(\Omega)} := \left( \| \partial_r v \|^2_{L^2_1(\Omega)} + \| \partial_z v \|^2_{L^2_1(\Omega)} \right)^{\frac{1}{2}},
\end{align*}
and norm
\begin{align*}
\| v \|_{H^1_1(\Omega)} := \left( \| v \|^2_{L^2_1(\Omega)} + | v |^2_{H^1_1(\Omega)} \right)^{\frac{1}{2}}.
\end{align*}
\begin{remark}
The definition can be extended in a natural way to $H^m_1(\Omega)$ for an arbitrary integer $m \geq 2$
and further, by interpolation, to $H^s_1(\Omega)$ for non-integer $s > 0$.
\end{remark}

We will also need the weighted space
\begin{align*}
V^1_1(\Omega) := H^1_1(\Omega) \cap L^2_{-1}(\Omega),
\end{align*}
equipped with the norm
\begin{align*}
\| v \|_{V^1_1(\Omega)} := \left( \| v \|^2_{L^2_{-1}(\Omega)} + | v |^2_{H^1_1(\Omega)} \right)^{\frac{1}{2}}.
\end{align*}
It can be proved \cite[Proposition 4.1]{mercier_raugel_1982} that all functions in $V^1_1(\Omega)$ have a null trace on
the part $\Gamma_0$ of the boundary contained in the $z$-axis.

We finally introduce the subspaces
\begin{equation}
\label{eq:defl210}
L^2_{1,0}(\Omega) := \left\{q | \, q \in L^2_1(\Omega), \;
  \int_{\Omega} q(r, z) \, r \, \text{d}r \, \text{d}z = 0 \right\},
\end{equation}
consisting of functions in $L^2_1(\Omega)$ with weighted integral equal to zero, and
\begin{align*}
H^1_{1 \scriptstyle{\diamond}}(\Omega) & :=
  \left\{v | \, v \in H^1_1(\Omega), \; v = 0 \text{ on } \Gamma \right\}, \\
V^1_{1 \scriptstyle{\diamond}}(\Omega) & :=
  \left\{v | \, v \in V^1_1(\Omega), \; v = 0 \text{ on } \Gamma \right\},
\end{align*}
consisting of functions in $H^1_1(\Omega)$ and $V^1_1(\Omega)$ that vanish on the part $\Gamma = \partial \Omega \backslash \Gamma_0$
of the boundary that is not contained in the $z$-axis.

All spaces defined above are Hilbert spaces for the inner products associated with the given norms.

\subsection{%
Characterization of
\texorpdfstring{$L^2(\breve{\Omega})$}{%
\ifpdfstringunicode{L\unichar{"00B2}(\unichar{"03A9})}{L2(Omega)}} and %
\texorpdfstring{$L^2_0(\breve{\Omega})$}{%
\ifpdfstringunicode{L\unichar{"00B2}\unichar{"2080}(\unichar{"03A9})}{L20(Omega)}}%
}
Recalling the identity~\eqref{eq:l2normcylcoord2Dsum}
we let the right-hand side terms, for all $k \in \mathbb{Z}$, define spaces $L_{(k)}^2(\Omega) := L^2_1(\Omega)$
for the Fourier coefficients $q^k$ of a function $\breve{q} \in L^2(\breve{\Omega})$, with norms
\begin{equation}
\label{eq:l2k_norm}
\| q^k \|_{L_{(k)}^2(\Omega)}^2 = \| q^k \|_{L^2_1(\Omega)}^2 = \int_{\Omega} | q^k |^2 \, r \, \text{d}r \, \text{d}z.
\end{equation}
Introducing the $l_2$-sum \cite[p$.$ 63]{MR2248303}
$\bigoplus_{2} \bigl\{ L_{(k)}^2(\Omega) \colon k \in \mathbb{Z} \bigr\}$, which is the subspace
of the Cartesian product $\times \bigl\{ L_{(k)}^2(\Omega) \colon k \in \mathbb{Z} \bigr\}$
consisting of all sequences $(q^k)_{k \in \mathbb{Z}}$ for which the norm
\begin{equation}
\label{eq:l2sum_of_l2k_norm}
\| (q^k) \|_{\bigoplus_{2} \{ L_{(k)}^2(\Omega) \colon k \in \mathbb{Z} \}} :=
\Bigl( \sum_{k \in \mathbb{Z}} \| q^k \|_{L_{(k)}^2(\Omega)}^2 \Bigr)^{\frac{1}{2}} < + \infty,
\end{equation}
and combining~\eqref{eq:l2normcylcoord2Dsum} with~\eqref{eq:l2k_norm}--\eqref{eq:l2sum_of_l2k_norm},
we obtain the following characterization of $L^2(\breve{\Omega})$:
\begin{theorem}
\label{theorem:charactl2}
The mapping
\begin{align*}
\breve{q} \mapsto (q^k)_{k \in \mathbb{Z}},
\end{align*}
defined by~\eqref{eq:q_cartesian_cyl} and~\eqref{eq:fourier_q}, is an isometric isomorphism between
$L^2(\breve{\Omega})$ and the $l_2$-sum
$\bigoplus_{2} \bigl\{ L_{(k)}^2(\Omega) \colon k \in \mathbb{Z} \bigr\}$:
\begin{equation}
\label{eq:l2_isometry}
\| \breve{q} \|_{L^2(\breve{\Omega})} =
\Bigl( \sum_{k \in \mathbb{Z}} \| q^k \|_{L_{(k)}^2(\Omega)}^2 \Bigr)^{\frac{1}{2}} =
\| (q^k) \|_{\bigoplus_{2} \{ L_{(k)}^2(\Omega) \colon k \in \mathbb{Z} \}},
\end{equation}
where, for all $k \in \mathbb{Z}$,
$L_{(k)}^2(\Omega) = L^2_1(\Omega)$.
\end{theorem}
Since, for $\breve{q} \in L^2(\breve{\Omega})$, by~\eqref{eq:q_cartesian_cyl} and~\eqref{eq:fourier_q},
\begin{align*}
\int_{\breve{\Omega}} \breve{q}(x, y, z) \, \text{d}x \, \text{d}y \, \text{d}z & =
\int_{\tilde{\Omega}} q(r, \theta, z) \, r \, \text{d}r \, \text{d}\theta \, \text{d}z \\
& = \frac{1}{\sqrt{2 \pi}} \, \sum_{k \in \mathbb{Z}}
  \int_{\tilde{\Omega}} q^k(r, z) e^{ik \theta} \, r \, \text{d}r \, \text{d}\theta \, \text{d}z \\
& = \sqrt{2 \pi} \, \int_{\Omega} q^0(r, z) \, r \, \text{d}r \, \text{d}z,
\end{align*}
we also, recalling~\eqref{eq:defl210}, obtain a characterization of the subspace $L^2_0(\breve{\Omega})$ of $L^2(\breve{\Omega})$:
\begin{corollary}
The mapping
\begin{align*}
\breve{q} \mapsto (q^k)_{k \in \mathbb{Z}},
\end{align*}
defined by~\eqref{eq:q_cartesian_cyl} and~\eqref{eq:fourier_q}, is an isometric isomorphism between
$L^2_0(\breve{\Omega})$ and the $l_2$-sum $\bigoplus_{2} \bigl\{ L_{(k),0}^2(\Omega) \colon k \in \mathbb{Z} \bigr\}$, where
\begin{align*}
L_{(k),0}^2(\Omega) := \left\{
\begin{array}{rcl}
L^2_{1,0}(\Omega), & \text{if} & k=0, \\
L^2_{1}(\Omega), & \text{if} & k \neq 0,
\end{array}
\right.
\end{align*}
and, for all $k \in \mathbb{Z}$, $\| q^k \|_{L_{(k),0}^2(\Omega)} = \| q^k \|_{L_1^2(\Omega)}$.
\end{corollary}

\subsection{%
Characterization of
\texorpdfstring{$\bigl(H^1(\breve{\Omega})\bigr)^3$}{%
\ifpdfstringunicode{(H\unichar{"00B9}(\unichar{"03A9}))\unichar{"00B3}}{(H1(Omega))3}} and %
\texorpdfstring{$\bigl(H^1_0(\breve{\Omega})\bigr)^3$}{%
\ifpdfstringunicode{(H\unichar{"00B9}\unichar{"2080}(\unichar{"03A9}))\unichar{"00B3}}{(H10(Omega))3}}%
}
Similarly, led by the right-hand side terms in the identity~\eqref{eq:H1vecnormcylcoord2Dsum}, for all $k \in \mathbb{Z}$, we define spaces
$\boldsymbol{H}_{(k)}^1(\Omega)$ for the Fourier coefficient triples $\vec{v}^k = \left(v_r^k, v_{\theta}^k, v_z^k\right)^T$
of a vector function $\vec{\breve{v}} \in \bigl(H^1(\breve{\Omega})\bigr)^3$, by the norms
\begin{equation}
\label{eq:H1veck_norm}
\begin{alignedat}{1}
\| \vec{v}^k \|_{\boldsymbol{H}_{(k)}^1(\Omega)}^2 & : =
\int_{\Omega}
\Bigl(
  | v_r^k |^2 + | v_{\theta}^k |^2 + | v_z^k |^2 \\
  & \qquad \quad + | \partial_r v_r^k |^2 + | \partial_z v_r^k |^2
      + \frac{1+k^2}{r^2} \, | v_r^k |^2 + \frac{2ik}{r^2} \, v_{\theta}^k \bar{v}_r^k \\
  & \qquad \quad + | \partial_r v_{\theta}^k |^2 + | \partial_z v_{\theta}^k |^2
      + \frac{1+k^2}{r^2} \, | v_{\theta}^k |^2 - \frac{2ik}{r^2} \, v_r^k \bar{v}_{\theta}^k \\
  & \qquad \quad + | \partial_r v_z^k |^2 + | \partial_z v_z^k |^2 + \frac{k^2}{r^2} \, | v_z^k |^2
\Bigr) \, r \, \text{d}r \, \text{d}z.
\end{alignedat}
\end{equation}
To see that~\eqref{eq:H1veck_norm} satisfies the properties of a norm, we consider the function $\vec{\breve{v}}^k \in (H^1(\breve{\Omega}))^3$
corresponding to a single Fourier coefficient $\vec{v}^k \in \boldsymbol{H}_{(k)}^1(\Omega)$:
\begin{align*}
\vec{\breve{v}}^k(x, y, z) := \frac{1}{\sqrt{2 \pi}} \, \mathcal{R}_{\theta} \vec{v}^k(r, z) e^{ik \theta},
\end{align*}
and, from~\eqref{eq:H1vecnormcylcoord2Dsum}, note that
\begin{align}
\label{eq:norm_single_mode}
\| \vec{v}^k \|_{\boldsymbol{H}_{(k)}^1(\Omega)} = \| \vec{\breve{v}}^k \|_{(H^1(\breve{\Omega}))^3}.
\end{align}
From~\eqref{eq:norm_single_mode} and the norm properties of $\bigl(H^1(\breve{\Omega})\bigr)^3$, we immediately get:
\begin{align*}
\| \vec{v}^k \|_{\boldsymbol{H}_{(k)}^1(\Omega)} & \geq 0, \\
\| \vec{v}^k \|_{\boldsymbol{H}_{(k)}^1(\Omega)} & = 0 \quad \Leftrightarrow \quad
   \| \vec{\breve{v}}^k \|_{(H^1(\breve{\Omega}))^3} = 0 \quad \Leftrightarrow \quad
   \vec{\breve{v}}^k = 0 \quad \Leftrightarrow \quad \vec{v}^k = 0, \\
\| c \vec{v}^k \|_{\boldsymbol{H}_{(k)}^1(\Omega)} & = \| c \vec{\breve{v}}^k \|_{(H^1(\breve{\Omega}))^3}
   = |c| \, \| \vec{\breve{v}}^k \|_{(H^1(\breve{\Omega}))^3}
   = |c| \, \| \vec{v}^k \|_{\boldsymbol{H}_{(k)}^1(\Omega)}, \\
\| \vec{v}^k + \vec{w}^k \|_{\boldsymbol{H}_{(k)}^1(\Omega)} & =
    \| \vec{\breve{v}}^k + \vec{\breve{w}}^k\|_{(H^1(\breve{\Omega}))^3} \\
    & \leq \| \vec{\breve{v}}^k \|_{(H^1(\breve{\Omega}))^3} + \| \vec{\breve{w}}^k \|_{(H^1(\breve{\Omega}))^3} \\
    & = \| \vec{v}^k \|_{\boldsymbol{H}_{(k)}^1(\Omega)} + \| \vec{w}^k \|_{\boldsymbol{H}_{(k)}^1(\Omega)}.
\end{align*}
Also from~\eqref{eq:norm_single_mode}, the completeness of $\boldsymbol{H}_{(k)}^1(\Omega)$ is a consequence of
the completeness of $(H^1(\breve{\Omega}))^3$.

We now characterize the spaces $\boldsymbol{H}_{(k)}^1(\Omega)$ in terms of the weighted spaces defined in
Section~\ref{subsec:weighted_Sobolev_spaces}.
Beginning with the case $k=0$, from~\eqref{eq:H1veck_norm}:
\begin{equation}
\label{eq:H1vec0_norm}
\| \vec{v}^0 \|_{\boldsymbol{H}_{(0)}^1(\Omega)}^2 = \| \vec{v}^0 \|_{(L^2_1(\Omega))^3}^2 +
\| v_r^0 \|_{V^1_1(\Omega)}^2 + \| v_{\theta}^0 \|_{V^1_1(\Omega)}^2 + | v_z^0 |_{H^1_1(\Omega)}^2,
\end{equation}
and we deduce that
\begin{align}
\label{eq:H1vec0_space}
\boldsymbol{H}_{(0)}^1(\Omega) = V^1_1(\Omega) \times V^1_1(\Omega) \times H^1_1(\Omega).
\end{align}
For the cases $k=\pm 1$,~\eqref{eq:H1veck_norm} gives
\begin{equation*}
\begin{alignedat}{1}
\| \vec{v}^{\pm 1} \|_{\boldsymbol{H}_{(\pm 1)}^1(\Omega)}^2 & =
  \| \vec{v}^{\pm 1} \|_{(L^2_1(\Omega))^3}^2 +
  | v_r^{\pm 1} |_{H^1_1(\Omega)}^2 + | v_{\theta}^{\pm 1} |_{H^1_1(\Omega)}^2 + \| v_z^{\pm 1} \|_{V^1_1(\Omega)}^2 \\
& \qquad + 2 \int_{\Omega} \Bigl(
| v_r^{\pm 1} |^2 + | v_{\theta}^{\pm 1} |^2 \pm i (
v_{\theta}^{\pm 1} \bar{v}_r^{\pm 1} - v_r^{\pm 1} \bar{v}_{\theta}^{\pm 1} )
\Bigr) \, \frac{1}{r} \, \text{d}r \, \text{d}z,
\end{alignedat}
\end{equation*}
and by noting that
\begin{align*}
| v_r^{\pm 1} \pm i v_{\theta}^{\pm 1} |^2 & =
( v_r^{\pm 1} \pm i v_{\theta}^{\pm 1} ) ( \bar{v}_r^{\pm 1} \mp i \bar{v}_{\theta}^{\pm 1} ) \\
& = | v_r^{\pm 1} |^2 + | v_{\theta}^{\pm 1} |^2 \pm i ( v_{\theta}^{\pm 1} \bar{v}_r^{\pm 1} - v_r^{\pm 1} \bar{v}_{\theta}^{\pm 1} ),
\end{align*}
we obtain
\begin{equation}
\label{eq:H1vec1_norm}
\begin{alignedat}{1}
\| \vec{v}^{\pm 1} \|_{\boldsymbol{H}_{(\pm 1)}^1(\Omega)}^2 & =
\| \vec{v}^{\pm 1} \|_{(L^2_1(\Omega))^3}^2 + 2 \, \| v_r^{\pm 1} \pm i v_{\theta}^{\pm 1} \|_{L^2_{-1}(\Omega)}^2 \\
& \qquad + | v_r^{\pm 1} |_{H^1_1(\Omega)}^2 + | v_{\theta}^{\pm 1} |_{H^1_1(\Omega)}^2 + \| v_z^{\pm 1} \|_{V^1_1(\Omega)}^2,
\end{alignedat}
\end{equation}
from which we deduce that, for $k = \pm 1$:
\begin{equation}
\label{eq:H1vec1_space}
\boldsymbol{H}_{(k)}^1(\Omega) = \left\{ \vec{v}^k | \,
\vec{v}^k \in H^1_1(\Omega) \times H^1_1(\Omega) \times V^1_1(\Omega), \;
v_r^k + ik v_{\theta}^k \in L^2_{-1}(\Omega) \right\}.
\end{equation}
When $|k| \geq 2$,~\eqref{eq:H1veck_norm} gives
\begin{equation}
\label{eq:H1vec2_norm}
\begin{alignedat}{1}
\| \vec{v}^k \|_{\boldsymbol{H}_{(k)}^1(\Omega)}^2 & =
\| \vec{v}^k \|_{(L^2_1(\Omega))^3}^2 + | \vec{v}^k |_{(H^1_1(\Omega))^3}^2 \\
& \qquad \quad + (1 + k^2) \left( \| v_r^k \|_{L^2_{-1}(\Omega)}^2 + \| v_{\theta}^k \|_{L^2_{-1}(\Omega)}^2 \right) \\
& \qquad \quad + k^2 \, \| v_z^k \|_{L^2_{-1}(\Omega)}^2 \\
& \qquad \quad + 2ik \int_{\Omega} \left( v_{\theta}^k \bar{v}_r^k - v_r^k \bar{v}_{\theta}^k \right) \, \frac{1}{r} \, \text{d}r \, \text{d}z.
\end{alignedat}
\end{equation}
Noting that
\begin{align*}
i \big( v_{\theta}^k \bar{v}_r^k - v_r^k \bar{v}_{\theta}^k \bigr) = 2 \operatorname{Im}\bigl( v_r^k \bar{v}_{\theta}^k \bigr),
\end{align*}
confirming that the right-hand sides are real for all $k \in \mathbb{Z}$, we can estimate the last term
\begin{align*}
\left| 2ik \int_{\Omega} \left( v_{\theta}^k \bar{v}_r^k - v_r^k \bar{v}_{\theta}^k \right) \, \frac{1}{r} \, \text{d}r \, \text{d}z \right| &
\leq 4 |k| \int_{\Omega} |v_r^k| \, |{v}_{\theta}^k| \, \frac{1}{r} \, \text{d}r \, \text{d}z \\
& \leq 2 |k| \, \left( \| v_r^k \|_{L^2_{-1}(\Omega)}^2 + \| v_{\theta}^k \|_{L^2_{-1}(\Omega)}^2 \right),
\end{align*}
resulting in the inequalities
\begin{equation}
\begin{alignedat}{1}
\label{eq:inequalities}
& \| \vec{v}^k \|_{(L^2_1(\Omega))^3}^2 + | \vec{v}^k |_{(H^1_1(\Omega))^3}^2 \\
& \qquad + (|k| - 1)^2 \left(\| v_r^k \|_{L^2_{-1}(\Omega)}^2 + \| v_{\theta}^k \|_{L^2_{-1}(\Omega)}^2 \right)
+ k^2 \, \| v_z^k \|_{L^2_{-1}(\Omega)}^2 \\
& \leq \; \| \vec{v}^k \|_{\boldsymbol{H}_{(k)}^1(\Omega)}^2 \\
& \leq \; \| \vec{v}^k \|_{(L^2_1(\Omega))^3}^2 + | \vec{v}^k |_{(H^1_1(\Omega))^3}^2 \\
& \qquad + (|k| + 1)^2 \left(\| v_r^k \|_{L^2_{-1}(\Omega)}^2 + \| v_{\theta}^k \|_{L^2_{-1}(\Omega)}^2 \right)
+ k^2 \, \| v_z^k \|_{L^2_{-1}(\Omega)}^2.
\end{alignedat}
\end{equation}
Since $|k|\geq 2$,
\begin{align*}
(|k| - 1)^2 & = \Bigl( 1 - \frac{1}{|k|} \Bigr)^2 \, k^2 \geq \frac{1}{4} \, k^2, \\
(|k| + 1)^2 & = \Bigl( 1 + \frac{1}{|k|} \Bigr)^2 \, k^2 \leq \frac{9}{4} \, k^2,
\end{align*}
which, combined with~\eqref{eq:inequalities}, shows that the norm $\| \cdot \|_{\boldsymbol{H}_{(k)}^1(\Omega)}$
is equivalent to the norm $\| \cdot \|_{\boldsymbol{H}_{(k)*}^1(\Omega)}$:
\begin{align*}
\frac{1}{2} \, \| \vec{v}^k \|_{\boldsymbol{H}_{(k)*}^1(\Omega)} \leq
\| \vec{v}^k \|_{\boldsymbol{H}_{(k)}^1(\Omega)} \leq
\frac{3}{2} \, \| \vec{v}^k \|_{\boldsymbol{H}_{(k)*}^1(\Omega)},
\end{align*}
where $\| \cdot \|_{\boldsymbol{H}_{(k)*}^1(\Omega)}$ is defined by
\begin{equation}
\label{eq:H1vec2_ekv_norm}
\begin{alignedat}{1}
\| \vec{v}^k \|_{\boldsymbol{H}_{(k)*}^1(\Omega)}^2 & :=
\| \vec{v}^k \|_{(L^2_1(\Omega))^3}^2 + | \vec{v}^k |_{(H^1_1(\Omega))^3}^2 + k^2 \| \vec{v}^k \|_{(L^2_{-1}(\Omega))^3}^2.
\end{alignedat}
\end{equation}
We deduce that, for $|k|\geq 2$,
\begin{align}
\label{eq:H1vec2_space}
\boldsymbol{H}_{(k)}^1(\Omega) = V^1_1(\Omega) \times V^1_1(\Omega) \times V^1_1(\Omega).
\end{align}
We summarize our results in the following characterization of $\bigl(H^1(\breve{\Omega})\bigr)^3$:
\begin{theorem}
\label{theorem:characth1}
The mapping
\begin{align*}
\vec{\breve{v}} \mapsto (\vec{v}^k)_{k \in \mathbb{Z}},
\end{align*}
defined by~\eqref{eq:vec_v_cartesian_cyl} and \eqref{eq:fourier_v},
is an isometric isomorphism between $\bigl(H^1(\breve{\Omega})\bigr)^3$ and the $l_2$-sum
$\bigoplus_{2} \bigl\{ \boldsymbol{H}_{(k)}^1(\Omega) \colon k \in \mathbb{Z} \bigr\}$:
\begin{align}
\label{eq:H1vec_isometry}
\| \vec{\breve{v}} \|_{(H^1(\breve{\Omega}))^3} =
  \Bigl( \sum_{k \in \mathbb{Z}} \| \vec{v}^k \|^2_{\boldsymbol{H}_{(k)}^1(\Omega)} \Bigr)^{\frac{1}{2}} =
  \| (\vec{v}^k) \|_{\bigoplus_{2} \{ \boldsymbol{H}_{(k)}^1(\Omega) \colon k \in \mathbb{Z} \}},
\end{align}
where, from~\eqref{eq:H1vec0_space},~\eqref{eq:H1vec1_space},~\eqref{eq:H1vec2_space}:
\begin{align*}
\boldsymbol{H}_{(k)}^1(\Omega) := \left\{
\begin{array}{lcl}
  V^1_1(\Omega) \times V^1_1(\Omega) \times H^1_1(\Omega), & \text{if} & k=0, \\
  \left\{ \vec{v}^k \in H^1_1(\Omega) \times H^1_1(\Omega) \times V^1_1(\Omega) \right., && \\
\qquad \qquad \qquad \left. v_r^k + ik v_{\theta}^k \in L^2_{-1}(\Omega) \right\},
  & \text{if} & |k|=1, \\
 V^1_1(\Omega) \times V^1_1(\Omega) \times V^1_1(\Omega), & \text{if} & |k| \geq 2,
\end{array}
\right.
\end{align*}
and where the norms $\| \cdot \|_{\boldsymbol{H}_{(k)}^1(\Omega)}$
are given by~\eqref{eq:H1vec0_norm},~\eqref{eq:H1vec1_norm},~\eqref{eq:H1vec2_norm}:
\begin{align*}
\| \vec{v}^k \|_{\boldsymbol{H}_{(k)}^1(\Omega)}^2 := \left\{
\begin{array}{lcl}
  \| \vec{v}^k \|_{(L^2_1(\Omega))^3}^2 & & \\ [1.5ex]
    \qquad + \| v_r^k \|_{V^1_1(\Omega)}^2 + \| v_{\theta}^k \|_{V^1_1(\Omega)}^2 + | v_z^k |_{H^1_1(\Omega)}^2, & \text{if} & k=0, \\ [2ex]
  \| \vec{v}^k \|_{(L^2_1(\Omega))^3}^2 + 2 \, \| v_r^k + ik v_{\theta}^k \|_{L^2_{-1}(\Omega)}^2 & & \\ [1.5ex]
    \qquad + | v_r^k |_{H^1_1(\Omega)}^2 + | v_{\theta}^k |_{H^1_1(\Omega)}^2 + \| v_z^k \|_{V^1_1(\Omega)}^2, & \text{if} & |k|=1, \\ [2ex]
  \| \vec{v}^k \|_{(L^2_1(\Omega))^3}^2 + | \vec{v}^k |_{(H^1_1(\Omega))^3}^2 & & \\ [1.5ex]
    \qquad + (1 + k^2) \left( \| v_r^k \|_{L^2_{-1}(\Omega)}^2 + \| v_{\theta}^k \|_{L^2_{-1}(\Omega)}^2 \right) & & \\ [1.5ex]
    \qquad + k^2 \, \| v_z^k \|_{L^2_{-1}(\Omega)}^2 & & \\ [1.5ex]
    \qquad + 2ik \int_{\Omega} \bigl( v_{\theta}^k \bar{v}_r^k - v_r^k \bar{v}_{\theta}^k \bigr) \, \frac{1}{r} \, \text{d}r \, \text{d}z,
    & \text{if} & |k| \geq 2.
\end{array}
\right.
\end{align*}
For $|k| \geq 2$, we can use the equivalent norm $\| \cdot  \|_{\boldsymbol{H}_{(k)*}^1(\Omega)}$
defined by~\eqref{eq:H1vec2_ekv_norm}:
\begin{align*}
\| \vec{v}^k  \|_{\boldsymbol{H}_{(k)*}^1(\Omega)}^2 & :=
\| \vec{v}^k \|_{(L^2_1(\Omega))^3}^2 + | \vec{v}^k |_{(H^1_1(\Omega))^3}^2 + k^2 \| \vec{v}^k \|_{(L^2_{-1}(\Omega))^3}^2.
\end{align*}
\end{theorem}
\begin{remark}
Corresponding to~\eqref{eq:H1vec_isometry}, we restate~\eqref{eq:H1vecinnerproductcylcoord2Dsum}
in terms of the inner products $( \cdot, \cdot )_{\boldsymbol{H}_{(k)}^1(\Omega)}$
associated with the norms $\| \cdot \|_{\boldsymbol{H}_{(k)}^1(\Omega)}$:
\begin{align}
\label{eq:H1_inner_prod_decomposition}
( \vec{\breve{v}}, \vec{\breve{w}} )_{(H^1(\breve{\Omega}))^3}
= \sum_{k \in \mathbb{Z}} ( \vec{v}^k, \vec{w}^k )_{\boldsymbol{H}_{(k)}^1(\Omega)}.
\end{align}
\end{remark}
\begin{remark}
For comparison, in Appendix~\ref{app:sec:comparison} we
derive~\eqref{eq:H1vec0_norm},~\eqref{eq:H1vec1_norm},~\eqref{eq:H1vec2_norm}
by an alternative method used in~\cite{2020arXiv200407216C}.
\end{remark}
We now consider the subspace $\bigl(H^1_0(\breve{\Omega})\bigr)^3$ of $\bigl(H^1(\breve{\Omega})\bigr)^3$. The corresponding
subspaces of $\boldsymbol{H}_{(k)}^1(\Omega)$ are
\begin{align*}
\boldsymbol{H}_{(k) \scriptstyle{\diamond}}^1(\Omega) :=
\boldsymbol{H}_{(k)}^1(\Omega) \cap \bigl( H^1_{1 \scriptstyle{\diamond}}(\Omega) \bigr)^3,
\end{align*}
consisting of vector functions in $\boldsymbol{H}_{(k)}^1(\Omega)$ that vanish on the part $\Gamma = \partial \Omega \backslash \Gamma_0$
of the boundary that is not contained in the $z$-axis. Since we use the semi-norm $| \cdot |_{(H^1(\breve{\Omega}))^3}$ as a norm,
equivalent to $\| \cdot \|_{(H^1(\breve{\Omega}))^3}$,
on $\bigl(H^1_0(\breve{\Omega})\bigr)^3$,
we state the following representation of $| \cdot |_{(H^1(\breve{\Omega}))^3}$,
which follows immediately from~\eqref{eq:H1vecnormcylcoord2Dsum} by omitting the
terms corresponding to the $\bigl(L^2(\breve{\Omega})\bigr)^3$-norm:
\begin{equation}
\label{eq:H1vecseminormcylcoord2Dsum}
\begin{alignedat}{1}
| \vec{\breve{v}} |_{(H^1(\breve{\Omega}))^3}^2 & = \sum_{k \in \mathbb{Z}} | \vec{v}^k |_{\boldsymbol{H}_{(k)}^1(\Omega)}^2,
\end{alignedat}
\end{equation}
where we define
\begin{equation}
\label{eq:H1veck_seminorm}
\begin{alignedat}{1}
| \vec{v}^k |_{\boldsymbol{H}_{(k)}^1(\Omega)}^2 & : =
\int_{\Omega}
\Bigl(
  | \partial_r v_r^k |^2 + | \partial_z v_r^k |^2
      + \frac{1+k^2}{r^2} \, | v_r^k |^2 + \frac{2ik}{r^2} \, v_{\theta}^k \bar{v}_r^k \\
  & \qquad \quad + | \partial_r v_{\theta}^k |^2 + | \partial_z v_{\theta}^k |^2
      + \frac{1+k^2}{r^2} \, | v_{\theta}^k |^2 - \frac{2ik}{r^2} \, v_r^k \bar{v}_{\theta}^k \\
  & \qquad \quad + | \partial_r v_z^k |^2 + | \partial_z v_z^k |^2 + \frac{k^2}{r^2} \, | v_z^k |^2
\Bigr) \, r \, \text{d}r \, \text{d}z.
\end{alignedat}
\end{equation}
\begin{remark}
The relation corresponding to~\eqref{eq:H1vecseminormcylcoord2Dsum}
between the associated inner products $( \cdot, \cdot )_{(H^1_0(\breve{\Omega}))^3}$
and $( \cdot, \cdot )_{\boldsymbol{H}_{(k) \scriptstyle{\diamond}}^1(\Omega)}$ is:
\begin{align}
\label{eq:H10_inner_prod_decomposition}
( \vec{\breve{v}}, \vec{\breve{w}} )_{(H^1_0(\breve{\Omega}))^3}
= \sum_{k \in \mathbb{Z}} ( \vec{v}^k, \vec{w}^k )_{\boldsymbol{H}_{(k) \scriptstyle{\diamond}}^1(\Omega)}.
\end{align}
\end{remark}
For all $k \in \mathbb{Z}$, $| \cdot |_{\boldsymbol{H}_{(k)}^1(\Omega)}$ is a norm,
equivalent to $\| \cdot \|_{\boldsymbol{H}_{(k)}^1(\Omega)}$,
on $\boldsymbol{H}_{(k) \scriptstyle{\diamond}}^1(\Omega)$.
This follows by again considering a function $\vec{\breve{v}}^k \in \bigl(H^1_0(\breve{\Omega})\bigr)^3$
corresponding to a single Fourier coefficient $\vec{v}^k \in \boldsymbol{H}_{(k) \scriptstyle{\diamond}}^1(\Omega)$:
\begin{align*}
\vec{\breve{v}}^k(x, y, z) := \frac{1}{\sqrt{2 \pi}} \, \mathcal{R}_{\theta} \vec{v}^k(r, z) e^{ik \theta},
\end{align*}
and noting that, by~\eqref{eq:H1vecseminormcylcoord2Dsum},
the equivalence between
$| \cdot |_{(H^1(\breve{\Omega}))^3}$ and $\| \cdot \|_{(H^1(\breve{\Omega}))^3}$
on $\bigl(H^1_0(\breve{\Omega})\bigr)^3$,
and~\eqref{eq:H1vec_isometry},
\begin{align*}
| \vec{v}^k |_{\boldsymbol{H}_{(k)}^1(\Omega)} = | \vec{\breve{v}}^k |_{(H^1(\breve{\Omega}))^3} \geq
c \, \| \vec{\breve{v}}^k \|_{(H^1(\breve{\Omega}))^3} = c \, \| \vec{v}^k \|_{\boldsymbol{H}_{(k)}^1(\Omega)}.
\end{align*}
We get the following characterization of $\bigl(H^1_0(\breve{\Omega})\bigr)^3$:
\begin{corollary}
The mapping
\begin{align*}
\vec{\breve{v}} \mapsto (\vec{v}^k)_{k \in \mathbb{Z}},
\end{align*}
defined by~\eqref{eq:vec_v_cartesian_cyl} and \eqref{eq:fourier_v},
is an isometric isomorphism between $\bigl(H^1_0(\breve{\Omega})\bigr)^3$ and the $l_2$-sum
$\bigoplus_{2} \bigl\{ \boldsymbol{H}_{(k) \scriptstyle{\diamond}}^1(\Omega) \colon k \in \mathbb{Z} \bigr\}$:
\begin{align}
\label{eq:H10vec_isometry}
| \vec{\breve{v}} |_{(H^1(\breve{\Omega}))^3} =
  \Bigl( \sum_{k \in \mathbb{Z}} | \vec{v}^k |^2_{\boldsymbol{H}_{(k)}^1(\Omega)} \Bigr)^{\frac{1}{2}} =
  \| (\vec{v}^k) \|_{\bigoplus_{2} \{ \boldsymbol{H}_{(k) \scriptstyle{\diamond}}^1(\Omega) \colon k \in \mathbb{Z} \}},
\end{align}
where
\begin{align*}
\boldsymbol{H}_{(k) \scriptstyle{\diamond}}^1(\Omega) := \left\{
\begin{array}{lcl}
  V^1_{1 \scriptstyle{\diamond}}(\Omega) \times
  V^1_{1 \scriptstyle{\diamond}}(\Omega) \times
  H^1_{1 \scriptstyle{\diamond}}(\Omega), & \text{if} & k=0, \\
  \left\{ \vec{v}^k \in
     H^1_{1 \scriptstyle{\diamond}}(\Omega) \times
     H^1_{1 \scriptstyle{\diamond}}(\Omega) \times
     V^1_{1 \scriptstyle{\diamond}}(\Omega), \right. && \\
  \qquad \qquad \qquad \; \left. v_r^k + ik v_{\theta}^k \in L^2_{-1}(\Omega) \right\},
  & \text{if} & |k|=1, \\
 V^1_{1 \scriptstyle{\diamond}}(\Omega) \times
 V^1_{1 \scriptstyle{\diamond}}(\Omega) \times
 V^1_{1 \scriptstyle{\diamond}}(\Omega), & \text{if} & |k| \geq 2,
\end{array}
\right.
\end{align*}
and
\begin{align*}
| \vec{v}^k |_{\boldsymbol{H}_{(k)}^1(\Omega)}^2 := \left\{
\begin{array}{lcl}
  \| v_r^k \|_{V^1_1(\Omega)}^2 + \| v_{\theta}^k \|_{V^1_1(\Omega)}^2 + | v_z^k |_{H^1_1(\Omega)}^2, & \text{if} & k=0, \\ [2ex]
  | v_r^k |_{H^1_1(\Omega)}^2 + | v_{\theta}^k |_{H^1_1(\Omega)}^2 + \| v_z^k \|_{V^1_1(\Omega)}^2 & & \\ [1.5ex]
    \qquad + 2 \, \| v_r^k + ik v_{\theta}^k \|_{L^2_{-1}(\Omega)}^2, & \text{if} & |k|=1, \\ [2ex]
    | \vec{v}^k |_{(H^1_1(\Omega))^3}^2 & & \\ [1.5ex]
    \qquad + (1 + k^2) \left( \| v_r^k \|_{L^2_{-1}(\Omega)}^2 + \| v_{\theta}^k \|_{L^2_{-1}(\Omega)}^2 \right) & & \\ [1.5ex]
    \qquad + k^2 \, \| v_z^k \|_{L^2_{-1}(\Omega)}^2 & & \\ [1.5ex]
    \qquad + 2ik \int_{\Omega} \bigl( v_{\theta}^k \bar{v}_r^k - v_r^k \bar{v}_{\theta}^k \bigr) \, \frac{1}{r} \, \text{d}r \, \text{d}z,
    & \text{if} & |k| \geq 2.
\end{array}
\right.
\end{align*}
For $|k| \geq 2$, we can use the equivalent norm $| \cdot  |_{\boldsymbol{H}_{(k)*}^1(\Omega)}$ defined by:
\begin{align*}
| \vec{v}^k  |_{\boldsymbol{H}_{(k)*}^1(\Omega)}^2 & :=
| \vec{v}^k |_{(H^1_1(\Omega))^3}^2 + k^2 \| \vec{v}^k \|_{(L^2_{-1}(\Omega))^3}^2.
\end{align*}
\end{corollary}

\subsection{%
Characterization of
\texorpdfstring{$\bigl(H^{-1}(\breve{\Omega})\bigr)^3$}{%
\ifpdfstringunicode{(H\unichar{"207B}\unichar{"00B9}(\unichar{"03A9}))\unichar{"00B3}}{(H-1(Omega))3}}%
}
For $\vec{\breve{f}} \in \bigl(H^{-1}(\breve{\Omega})\bigr)^3$, from the Riesz representation theorem
we deduce the existence of (a unique) $\vec{\breve{w}}_f \in \bigl(H^1_0(\breve{\Omega})\bigr)^3$,
with $\| \vec{\breve{f}} \|_{(H^{-1}(\breve{\Omega}))^3} = | \vec{\breve{w}}_f |_{(H^1(\breve{\Omega}))^3}$,
such that for all $\vec{\breve{v}} \in \bigl(H^1_0(\breve{\Omega})\bigr)^3$:
\begin{equation}
\label{eq:Hminus1fourier}
\begin{alignedat}{1}
\langle \vec{\breve{f}}, \vec{\breve{v}} \rangle_{(H^{-1}(\breve{\Omega}))^3 \times (H^1_0(\breve{\Omega}))^3}
& = ( \vec{\breve{w}}_f, \vec{\breve{v}} )_{(H^1_0(\breve{\Omega}))^3} \\
& = \sum_{k \in \mathbb{Z}} ( \vec{w}^k_f, \vec{v}^k )_{\boldsymbol{H}_{(k) \scriptstyle{\diamond}}^1(\Omega)} \\
& := \sum_{k \in \mathbb{Z}}
\langle \vec{f}^k, \vec{v}^k \rangle_{\boldsymbol{H}_{(k)}^{-1}(\Omega) \times \boldsymbol{H}_{(k) \scriptstyle{\diamond}}^1(\Omega)},
\end{alignedat}
\end{equation}
where we have used~\eqref{eq:H10_inner_prod_decomposition}, and
$\boldsymbol{H}_{(k)}^{-1}(\Omega)$ denotes the dual space of $\boldsymbol{H}_{(k) \scriptstyle{\diamond}}^1(\Omega)$.
\begin{remark}
For $\vec{\breve{f}} \in \bigl(L^{2}(\breve{\Omega})\bigr)^3$ we have:
\begin{align*}
\vec{f}^k(r, z) = \frac{1}{\sqrt{2 \pi}} \, \int_{-\pi}^{\pi} \bigl( \mathcal{R}_{-\theta} \vec{\breve{f}} \bigr) (r, \theta, z) e^{-ik \theta} \, \emph{d} \theta,
\end{align*}
and we can write the duality pairings as integrals:
\begin{align}
\langle \vec{\breve{f}}, \vec{\breve{v}} \rangle_{(H^{-1}(\breve{\Omega}))^3 \times (H^1_0(\breve{\Omega}))^3} & =
\int_{\breve{\Omega}}
\bigl( \vec{\breve{f}} \cdot \vec{\bar{\breve{v}}}
\bigr) \, \emph{d}x \, \emph{d}y \, \emph{d}z, \nonumber \\
\langle \vec{f}^k, \vec{v}^k \rangle_{\boldsymbol{H}_{(k)}^{-1}(\Omega) \times \boldsymbol{H}_{(k) \scriptstyle{\diamond}}^1(\Omega)} & =
\int_{\Omega}
\bigl( \vec{f}^k \cdot \vec{\bar{v}}^k
\bigr) \, r \, \emph{d}r \, \emph{d}z. \label{eq:duality_pairing_as_integral}
\end{align}
\end{remark}
From~\eqref{eq:H10vec_isometry} and the definition~\eqref{eq:Hminus1fourier} of $\vec{f}^k$ we get:
\begin{align*}
\| \vec{\breve{f}} \|^2_{(H^{-1}(\breve{\Omega}))^3} = | \vec{\breve{w}}_f |_{(H^1(\breve{\Omega}))^3}^2
= \sum_{k \in \mathbb{Z}} | \vec{w}^k_f |_{\boldsymbol{H}_{(k)}^1(\Omega)}^2
= \sum_{k \in \mathbb{Z}} \| \vec{f}^k \|_{\boldsymbol{H}_{(k)}^{-1}(\Omega)}^2,
\end{align*}
which gives the following characterization of $\bigl(H^{-1}(\breve{\Omega})\bigr)^3$:
\begin{theorem}
\label{theorem:characthm1}
The mapping
\begin{align*}
\vec{\breve{f}} \mapsto (\vec{f}^k)_{k \in \mathbb{Z}},
\end{align*}
defined by~\eqref{eq:Hminus1fourier}, is an isometric isomorphism between $\bigl(H^{-1}(\breve{\Omega})\bigr)^3$ and the $l_2$-sum
$\bigoplus_{2} \bigl\{ \boldsymbol{H}_{(k)}^{-1}(\Omega) \colon k \in \mathbb{Z} \bigr\}$:
\begin{align}
\label{eq:Hminus1vec_isometry}
\| \vec{\breve{f}} \|_{(H^{-1}(\breve{\Omega}))^3} =
  \Bigl( \sum_{k \in \mathbb{Z}} \| \vec{f}^k \|^2_{\boldsymbol{H}_{(k)}^{-1}(\Omega)} \Bigr)^{\frac{1}{2}} =
  \| (\vec{f}^k) \|_{\bigoplus_{2} \{ \boldsymbol{H}_{(k)}^{-1}(\Omega) \colon k \in \mathbb{Z} \}}.
\end{align}
\end{theorem}

\section{Two-dimensional problems}
\label{sec:two-dimensional_problems}
In this section we state variational formulations of the two-dimensional problems for the Fourier coefficients,
which are set in the spaces from the previous section. We show existence, uniqueness and stability from inf-sup conditions.
We conclude by discussing some features of the special case with axisymmetric data.

\subsection{Variational formulation}
\label{subsec:variational_formulation}
Introducing, for all $k \in \mathbb{Z}$, the sesquilinear forms
\begin{align*}
\mathcal{A}_k(\vec{u}, \vec{v}) := & (\vec{u}, \vec{v} )_{\boldsymbol{H}_{(k) \scriptstyle{\diamond}}^1(\Omega)} \\
= & \int_{\Omega}
\Bigl(
  \partial_r u_r \partial_r \bar{v}_r + \partial_z u_r \partial_z \bar{v}_r
      + \frac{1+k^2}{r^2} \, u_r \bar{v}_r + \frac{2ik}{r^2} \, u_{\theta} \bar{v}_r \\
  & \qquad + \partial_r u_{\theta} \partial_r \bar{v}_{\theta} + \partial_z u_{\theta} \partial_z \bar{v}_{\theta}
      + \frac{1+k^2}{r^2} \, u_{\theta} \bar{v}_{\theta} - \frac{2ik}{r^2} \, u_r \bar{v}_{\theta} \\
  & \qquad + \partial_r u_z \partial_r \bar{v}_z + \partial_z u_z \partial_z \bar{v}_z + \frac{k^2}{r^2} \, u_z \bar{v}_z
\Bigr) \, r \, \text{d}r \, \text{d}z, \\
\mathcal{B}_k(\vec{v}, q) := & - \int_{\Omega} ( \text{div}_k \, \vec{v} ) \bar{q} r \, \text{d}r \, \text{d}z \\
= & - \int_{\Omega} \bigl( \partial_r ( r v_r ) + ik v_{\theta} + \partial_z (r v_z )
\bigr) \bar{q} \, \text{d}r \, \text{d}z,
\end{align*}
where $\text{div}_k$ is defined by~\eqref{eq:divk},
we consider the following variational formulation of~\eqref{eq:stokes2dproblems}:

Find $( \vec{u}^k, p^k ) \in \boldsymbol{H}_{(k)}^1(\Omega) \times L_{(k),0}^2(\Omega)$,
where $\vec{u}^k - \vec{g}^k \in \boldsymbol{H}_{(k) \scriptstyle{\diamond}}^1(\Omega)$, such that
for all $(\vec{v} , q) \in \boldsymbol{H}_{(k) \scriptstyle{\diamond}}^1(\Omega) \times L_{(k),0}^2(\Omega)$
\begin{equation}
\label{eq:stokes2dvariational}
\left\{
\begin{alignedat}{1}
\mathcal{A}_k(\vec{u}^k, \vec{v}) + \overline{\mathcal{B}_k(\vec{v}, p^k)} & =
\langle \vec{f}^k, \vec{v} \rangle_{\boldsymbol{H}_{(k)}^{-1}(\Omega) \times \boldsymbol{H}_{(k) \scriptstyle{\diamond}}^1(\Omega)}, \\
\mathcal{B}_k(\vec{u}^k, q) & = 0.
\end{alignedat}
\right.
\end{equation}
\begin{remark}
Note that the term $\mathcal{B}_k$ is conjugated in the first equation in~\eqref{eq:stokes2dvariational},
making the sesquilinear form
\begin{align*}
\mathcal{M}_k \! \left((\vec{u}, p),(\vec{v}, q)\right) :=
\mathcal{A}_k(\vec{u}, \vec{v}) + \overline{\mathcal{B}_k(\vec{v}, p)} + \mathcal{B}_k(\vec{u}, q)
\end{align*}
Hermitian.
\end{remark}
\begin{remark}
We use the same notation $\vec{g}^k$ for
the Fourier coefficients of the Dirichlet boundary data $\vec{\breve{g}}$,
and their liftings in $\boldsymbol{H}_{(k)}^1(\Omega)$.
Existence of these liftings follows from the three-dimensional trace theorem since
$\vec{\breve{g}} \in \bigl(H^{\frac{1}{2}}(\partial \breve{\Omega})\bigr)^3$
and thus, again using the same notation, admits a lifting $\vec{\breve{g}} \in \bigl(H^1(\breve{\Omega})\bigr)^3$.
\end{remark}
\begin{remark}
For $\vec{\breve{f}} \in \bigl(L^{2}(\breve{\Omega})\bigr)^3$ we can, as noted in~\eqref{eq:duality_pairing_as_integral}, write the duality pairings as integrals:
\begin{align*}
\langle \vec{f}^k, \vec{v} \rangle_{\boldsymbol{H}_{(k)}^{-1}(\Omega) \times \boldsymbol{H}_{(k) \scriptstyle{\diamond}}^1(\Omega)} & =
\int_{\Omega}
\bigl( \vec{f}^k \cdot \vec{\bar{v}}
\bigr) \, r \, \emph{d}r \, \emph{d}z.
\end{align*}
\end{remark}
\begin{remark}
The sesquilinear forms $\mathcal{A}_k(\cdot, \cdot)$ and $\mathcal{B}_k(\cdot, \cdot)$
are related to the corresponding sesquilinear forms
\begin{align*}
\breve{a}(\vec{\breve{u}}, \vec{\breve{v}}) &
= ( \vec{\breve{u}}, \vec{\breve{v}} )_{(H^1_0(\breve{\Omega}))^3}
= \int_{\breve{\Omega}} \gradtensorremark \vec{\breve{u}} : \gradtensorremark \vec{\bar{\breve{v}}} \; \emph{d}x \, \emph{d}y \, \emph{d}z, \\
\breve{b}(\vec{\breve{v}}, \breve{q}) & =  - \int_{\breve{\Omega}} ( \divremark \vec{\breve{v}} ) \bar{\breve{q}} \, \emph{d}x \, \emph{d}y \, \emph{d}z,
\end{align*}
in the mixed formulation of the three-dimensional problem~\eqref{eq:stokesproblem}:

Find $( \vec{\breve{u}}, \breve{p} ) \in \bigl(H^1(\breve{\Omega})\bigr)^3 \times L^2_0(\breve{\Omega})$,
where $\vec{\breve{u}} - \vec{\breve{g}} \in \bigl(H^1_0(\breve{\Omega})\bigr)^3$, such that
for all $(\vec{\breve{v}} , \breve{q}) \in \bigl(H^1_0(\breve{\Omega})\bigr)^3 \times L^2_0(\breve{\Omega})$
\begin{equation*}
\left\{
\begin{alignedat}{1}
\breve{a}(\vec{\breve{u}}, \vec{\breve{v}}) +
\overline{\breve{b}(\vec{\breve{v}}, \breve{p})} & =
\langle \vec{\breve{f}}, \vec{\breve{v}} \rangle_{(H^{-1}(\breve{\Omega}))^3 \times (H^1_0(\breve{\Omega}))^3}, \\
\breve{b}(\vec{\breve{u}}, \breve{q}) & = 0,
\end{alignedat}
\right.
\end{equation*}
by the relations
\begin{align}
\label{eq:sesquiA2d3d}
\breve{a}(\vec{\breve{u}}, \vec{\breve{v}}) & = \sum_{k \in \mathbb{Z}} \mathcal{A}_k(\vec{u}^k, \vec{v}^k), \\
\label{eq:sesquiB2d3d}
\breve{b}(\vec{\breve{v}}, \breve{q}) & = \sum_{k \in \mathbb{Z}} \mathcal{B}_k(\vec{v}^k, q^k),
\end{align}
where~\eqref{eq:sesquiA2d3d} is a restatement of~\eqref{eq:H10_inner_prod_decomposition}, and~\eqref{eq:sesquiB2d3d}
is an analogous consequence of~\eqref{eq:divvcyl} and the orthogonality on $[-\pi, \pi]$
of the family $\{ e^{ik \theta} \}_{k = -\infty}^{+\infty}$ of basis functions (cf$.$~\eqref{eq:parseval1}).
\end{remark}

\subsection{Well-posedness}
\label{sec:existence_and_uniqueness}
Setting $\vec{u}^k = \vec{u}^k_0 + \vec{g}^k$, we rewrite~\eqref{eq:stokes2dvariational}:

Find $( \vec{u}^k_0, p^k ) \in \boldsymbol{H}_{(k) \scriptstyle{\diamond}}^1(\Omega) \times L_{(k),0}^2(\Omega)$ such that
for all $(\vec{v} , q) \in \boldsymbol{H}_{(k) \scriptstyle{\diamond}}^1(\Omega) \times L_{(k),0}^2(\Omega)$
\begin{equation*}
\left\{
\begin{alignedat}{1}
\mathcal{A}_k(\vec{u}^k_0, \vec{v}) + \overline{\mathcal{B}_k(\vec{v}, p^k)} & =
\langle \vec{f}^k, \vec{v} \rangle_{\boldsymbol{H}_{(k)}^{-1}(\Omega) \times \boldsymbol{H}_{(k) \scriptstyle{\diamond}}^1(\Omega)} -
\mathcal{A}_k(\vec{g}^k, \vec{v}), \\
\mathcal{B}_k(\vec{u}^k_0, q) & =
-\mathcal{B}_k(\vec{g}^k, q).
\end{alignedat}
\right.
\end{equation*}
Since $\mathcal{A}_k(\cdot, \cdot)$ is a nonnegative, Hermitian form,
where $\mathcal{A}_k(\vec{v}, \vec{v}) = | \vec{v} |_{\boldsymbol{H}_{(k)}^1(\Omega)}^2$,
we have $| \mathcal{A}_k(\vec{u}, \vec{v}) | \leq
\mathcal{A}_k(\vec{u}, \vec{u})^{\frac{1}{2}}
\mathcal{A}_k(\vec{v}, \vec{v})^{\frac{1}{2}} =
| \vec{u} |_{\boldsymbol{H}_{(k)}^1(\Omega)}
| \vec{v} |_{\boldsymbol{H}_{(k)}^1(\Omega)}$,
showing coercivity of $\mathcal{A}_k(\cdot, \cdot)$ on
$\boldsymbol{H}_{(k) \scriptstyle{\diamond}}^1(\Omega)$,
continuity of $\mathcal{A}_k(\cdot, \cdot)$ on
$\boldsymbol{H}_{(k) \scriptstyle{\diamond}}^1(\Omega) \times \boldsymbol{H}_{(k) \scriptstyle{\diamond}}^1(\Omega)$,
and $\mathcal{A}_k(\vec{g}^k, \cdot) \in \boldsymbol{H}_{(k)}^{-1}(\Omega)$ with
$\| \mathcal{A}_k(\vec{g}^k, \cdot) \|_{\boldsymbol{H}_{(k)}^{-1}(\Omega)} \leq
| \vec{g}^k |_{\boldsymbol{H}_{(k)}^1(\Omega)}$.

Again considering functions corresponding to a single Fourier coefficient
\begin{align}
\vec{\breve{v}}^k(x, y, z) & = \frac{1}{\sqrt{2 \pi}} \, \mathcal{R}_{\theta} \vec{v}^k(r, z) e^{ik \theta}, \nonumber \\
\breve{q}^k(x, y, z) & = \frac{1}{\sqrt{2 \pi}} \, q^k(r, z) e^{ik \theta}, \label{eq:single_mode_q}
\end{align}
we note from~\eqref{eq:sesquiB2d3d},
the continuity of the three-dimensional form $\breve{b}(\cdot, \cdot)$,~\eqref{eq:H1vecseminormcylcoord2Dsum},
and~\eqref{eq:l2_isometry} that
\begin{align*}
| \mathcal{B}_k(\vec{v}^k, q^k) | = |\breve{b}(\vec{\breve{v}}^k, \breve{q}^k)| & \leq
    \sqrt{3} \; | \vec{\breve{v}}^k |_{(H^1(\breve{\Omega}))^3} \| \breve{q}^k \|_{L^2(\breve{\Omega})} \\
& = \sqrt{3} \; | \vec{v}^k |_{\boldsymbol{H}_{(k)}^1(\Omega)} \| q^k \|_{L_1^2(\Omega)},
\end{align*}
which shows continuity of $\mathcal{B}_k(\cdot, \cdot)$ on
$\boldsymbol{H}_{(k) \scriptstyle{\diamond}}^1(\Omega) \times L_{(k),0}^2(\Omega)$, and
$\mathcal{B}_k(\vec{g}^k, \cdot) \in L_{(k),0}^2(\Omega)^*$
with $\| \mathcal{B}_k(\vec{g}^k, \cdot) \|_{L_{(k),0}^2(\Omega)^*} \leq
\sqrt{3} \; | \vec{g}^k |_{\boldsymbol{H}_{(k)}^1(\Omega)}$.

The inf-sup condition for $\mathcal{B}_k(\cdot, \cdot)$ on
$\boldsymbol{H}_{(k) \scriptstyle{\diamond}}^1(\Omega) \times L_{(k),0}^2(\Omega)$: There exists a positive
constant $\beta$ (independent of $k$) such that
\begin{align*}
\forall q^k \in L_{(k),0}^2(\Omega): \sup_{\vec{v}^k \in \boldsymbol{H}_{(k) \scriptstyle{\diamond}}^1(\Omega)}
\frac{|\mathcal{B}_k(\vec{v}^k, q^k)|}{| \vec{v}^k |_{\boldsymbol{H}_{(k)}^1(\Omega)}} \geq \beta \| q^k \|_{L^2_1(\Omega)},
\end{align*}
follows from the three-dimensional inf-sup condition, from which,
for an arbitrary $q^k \in L_{(k),0}^2(\Omega)$ and corresponding $\breve{q}^k \in L^2_0 (\breve{\Omega})$ given by~\eqref{eq:single_mode_q},
we deduce that
$\exists \vec{\breve{w}} \in \bigl(H^1_0(\breve{\Omega})\bigr)^3$
such that, for its $k$:th Fourier coefficient $\vec{w}^k \in \boldsymbol{H}_{(k) \scriptstyle{\diamond}}^1(\Omega)$:
\begin{align*}
\frac{|\mathcal{B}_k(\vec{w}^k, q^k)|}{| \vec{w}^k |_{\boldsymbol{H}_{(k)}^1(\Omega)}} \geq
\frac{| \mathcal{B}_k(\vec{w}^k, q^k) |}{| \vec{\breve{w}} |_{(H^1(\breve{\Omega}))^3}} =
\frac{| \breve{b}(\vec{\breve{w}}, \breve{q}^k )|}{| \vec{\breve{w}} |_{(H^1(\breve{\Omega}))^3}}
\geq \beta \| \breve{q}^k \|_{L^2(\breve{\Omega})}
= \beta \| q^k \|_{L^2_1{\Omega})}.
\end{align*}
We have thus proved (see, e.g.,~\cite[Theorem 4.2.3]{MR3097958}) the following
\begin{theorem}
For $\vec{\breve{f}} \in \bigl(H^{-1}(\breve{\Omega})\bigr)^3$ and $\vec{\breve{g}} \in \bigl(H^{\frac{1}{2}}(\partial \breve{\Omega})\bigr)^3$,
where $\vec{g}^0$ satisfies the compatibility condition~\eqref{eq:zeroflux2d},
the variational formulations~\eqref{eq:stokes2dvariational}
have unique solutions for all $k \in \mathbb{Z}$. The solutions are bounded, uniformly in $k$, by:
\begin{align}
\label{eq:regularitystokes2D}
\| \vec{u}^k \|_{\boldsymbol{H}_{(k)}^1(\Omega)} + \|p^k\|_{L^2_1 (\Omega)} \leq
C ( \| \vec{f}^k \|_{\boldsymbol{H}_{(k)}^{-1}(\Omega)} + \| \vec{g}^k \|_{\boldsymbol{H}_{(k)}^{1}(\Omega)} ).
\end{align}
\end{theorem}

\subsection{Axisymmetric data}
For axisymmetric data, i$.$e$.$, only $\vec{f}^0$ and $\vec{g}^0$ are non-zero, it follows immediately from
uniqueness that only $\vec{u}^0$ and $p^0$ are non-zero, which means that the solution will also be axisymmetric.
Also note that problem~\eqref{eq:stokes2dvariational} for $k=0$ decouples into two independent problems: one for
$(u_r^0, u_z^0, p^0)$ and one for $u_{\theta}^0$. If the data are real-valued, these problems have
real-valued solutions.

\section{Anisotropic spaces}
\label{sec:anisotropic_spaces}
We introduce two families $\boldsymbol{H}^{\pm 1,s}(\breve{\Omega})$ of anisotropic spaces,
where $s$ is a nonnegative real number measuring ``extra regularity'' in the angular direction,
that we (following~\cite{MR1693480}) will use to derive an estimate of the error
due to Fourier truncation in Section~\ref{sec:Fourier_truncation_error}.

When $s$ is a nonnegative integer, we define
\begin{align}
\label{eq:anisotropicscalarhpm1}
\boldsymbol{H}^{\pm 1,s}(\breve{\Omega}) & := \left\{\vec{\breve{v}} \vert \;
\mathcal{R}_{\theta} \partial^l_{\theta} \mathcal{R}_{-\theta} \vec{\breve{v}} \in \bigl( H^{\pm 1}(\breve{\Omega}) \bigr)^3, \; 0 \leq l \leq s \right\},
\end{align}
equipped with the norms
\begin{align}
\label{eq:anisotropic_hpm1_norm}
\| \vec{\breve{v}} \|_{\boldsymbol{H}^{\pm 1, s}(\breve{\Omega})} & :=
\Bigl( \sum_{k \in \mathbb{Z}} (1 + k^2)^s \| \vec{v}^k \|_{\boldsymbol{H}_{(k)}^{\pm 1}(\Omega)}^2 \Bigr)^{\frac{1}{2}}.
\end{align}
Note that $\boldsymbol{H}^{\pm 1,0}(\breve{\Omega}) = \bigl( H^{\pm 1}(\breve{\Omega}) \bigr)^3$.
The norms defined by~\eqref{eq:anisotropic_hpm1_norm}
are equivalent to the natural norms
\begin{align}
\label{eq:anisotropic_natural_hpm1_norm}
\Bigl(
  \sum_{l = 0}^s \| \mathcal{R}_{\theta} \partial^l_{\theta} \mathcal{R}_{-\theta} \vec{\breve{v}} \|_{( H^{\pm 1}(\breve{\Omega}))^3}^2 \Bigr)^{\frac{1}{2}},
\end{align}
as we show in the following
\begin{lemma}
\label{lemma:charactanisotropic2}
For any nonnegative integer $s$, the norms defined by~\eqref{eq:anisotropic_hpm1_norm}
are equivalent to the natural norms defined by~\eqref{eq:anisotropic_natural_hpm1_norm}.
\end{lemma}
\begin{proof}
Let $\vec{\breve{v}} \in \boldsymbol{H}^{\pm 1,s}(\breve{\Omega})$ and $0 \leq l \leq s$.
From the Fourier expansions
\begin{align*}
\mathcal{R}_{\theta} \partial^l_{\theta} \mathcal{R}_{-\theta} \vec{\breve{v}} (x, y, z) & =
  \frac{1}{\sqrt{2 \pi}} \, \sum_{k \in \mathbb{Z}} \mathcal{R}_{\theta} (ik)^l \vec{v}^k(r, z) e^{ik \theta},
\end{align*}
and Theorem~\ref{theorem:characth1}/\ref{theorem:characthm1}
(for $\boldsymbol{H}^{1, s}(\breve{\Omega})$ and $\boldsymbol{H}^{-1, s}(\breve{\Omega})$, respectively), we get
\begin{align*}
  \sum_{l = 0}^s \| \mathcal{R}_{\theta} \partial^l_{\theta} \mathcal{R}_{-\theta} \vec{\breve{v}} \|_{( H^{\pm 1}(\breve{\Omega}))^3}^2
  & = \sum_{l = 0}^s \sum_{k \in \mathbb{Z}}
       \| (ik)^l \vec{v}^k \|_{\boldsymbol{H}_{(k)}^{\pm 1}(\Omega)}^2 \\
  & = \sum_{k \in \mathbb{Z}} \Bigl( \sum_{l = 0}^s k^{2l} \Bigr) \| \vec{v}^k \|_{\boldsymbol{H}_{(k)}^{\pm 1}(\Omega)}^2,
\end{align*}
from which the result follows.
\end{proof}
We generalize by extending the norms defined by~\eqref{eq:anisotropic_hpm1_norm} to an arbitrary nonnegative real number $s$:
\begin{align}
\label{eq:anisotropic_hpm1_norm_s_arbitrary}
\boldsymbol{H}^{\pm 1,s}(\breve{\Omega}) & := \left\{\vec{\breve{v}} \vert \,
\| \vec{\breve{v}} \|_{\boldsymbol{H}^{\pm 1, s}(\breve{\Omega})} :=
\Bigl( \sum_{k \in \mathbb{Z}} (1 + k^2)^s \| \vec{v}^k \|_{\boldsymbol{H}_{(k)}^{\pm 1}(\Omega)}^2 \Bigr)^{\frac{1}{2}} < + \infty \right\}
\end{align}

\section{Error due to Fourier truncation}
\label{sec:Fourier_truncation_error}
Introducing the truncated Fourier series
\begin{align}
\label{eq:fourierutrunc}
\vec{\breve{u}}_{[N]} & = \frac{1}{\sqrt{2 \pi}} \, \sum_{|k| \leq N} \mathcal{R}_{\theta} \vec{u}^k(r, z) e^{ik \theta}, \\
\label{eq:fourierptrunc}
\breve{p}_{[N]} & = \frac{1}{\sqrt{2 \pi}} \, \sum_{|k| \leq N } p^k(r, z) e^{ik \theta},
\end{align}
we prove the following estimate of the error due to Fourier truncation:
\begin{theorem}
Let $s$ be a nonnegative real number. If the data $( \vec{\breve{f}}, \vec{\breve{g}} )$
of problem~\eqref{eq:stokesproblem} belong to $\boldsymbol{H}^{-1, s}(\breve{\Omega}) \times \boldsymbol{H}^{1,s}(\breve{\Omega})$,
the following error estimate holds between the solution $(\vec{\breve{u}}, \breve{p} )$
and its truncated Fourier series $(\vec{\breve{u}}_{[N]}, \breve{p}_{[N]})$:
\begin{align*}
\| \vec{\breve{u}} - \vec{\breve{u}}_{[N]} \|_{(H^1(\breve{\Omega}))^3} +
\|\breve{p} - \breve{p}_{[N]}\|_{L^2(\breve{\Omega})} \leq
C N^{-s} \bigl( \| \vec{\breve{f}}\|_{\boldsymbol{H}^{-1, s}(\breve{\Omega}) } +
\| \vec{\breve{g}}\|_{\boldsymbol{H}^{1, s}(\breve{\Omega}) } \bigr).
\end{align*}
\end{theorem}
\begin{proof}
Using the isometries~\eqref{eq:l2_isometry} and~\eqref{eq:H1vec_isometry},
together with the regularity estimate~\eqref{eq:regularitystokes2D}
and the definitions of the anisotropic norms~\eqref{eq:anisotropic_hpm1_norm_s_arbitrary},
we have:
\begin{align*}
\| \vec{\breve{u}} - \vec{\breve{u}}_{[N]} \|&_{(H^1(\breve{\Omega}))^3}^2 +
\|\breve{p} - \breve{p}_{[N]}\|_{L^2(\breve{\Omega})}^2 \\
& = \Big\| \frac{1}{\sqrt{2 \pi}} \, \sum_{|k| > N} \mathcal{R}_{\theta} \vec{u}^k e^{ik \theta} \Big\|_{(H^1(\breve{\Omega}))^3}^2 +
\Big\| \frac{1}{\sqrt{2 \pi}} \, \sum_{|k| > N } p^k e^{ik \theta} \Big\|_{L^2(\breve{\Omega})}^2 \\
& = \sum_{|k| > N} \bigl( \| \vec{u}^k \|_{\boldsymbol{H}_{(k)}^{1}(\Omega)}^2 + \|p^k\|^2_{L^2_1(\Omega)} \bigr) \\
& \leq 2C^2 \sum_{|k| > N} \bigl( \| \vec{f}^k \|_{\boldsymbol{H}_{(k)}^{-1}(\Omega)}^2 + \| \vec{g}^k \|_{\boldsymbol{H}_{(k)}^{1}(\Omega)}^2 \bigr)  \\
& \leq 2C^2 N^{-2s} \sum_{|k| > N} (1 + k^2)^s \bigl( \| \vec{f}^k \|_{\boldsymbol{H}_{(k)}^{-1}(\Omega)}^2 + \| \vec{g}^k \|_{\boldsymbol{H}_{(k)}^{1}(\Omega)}^2 \bigr)  \\
& \leq 2 C^2 N^{-2s} \bigl( \| \vec{\breve{f}} \|_{\boldsymbol{H}^{-1, s}(\breve{\Omega})}^2 + \| \vec{\breve{g}} \|_{\boldsymbol{H}^{1, s}(\breve{\Omega})}^2 \bigr).
\end{align*}
\end{proof}
\begin{remark}
As emphasized in~\cite{MR2210089}, the error from Fourier truncation only
depends on the regularity of the data $( \vec{\breve{f}}, \vec{\breve{g}} )$;
not on the regularity of the solution
$(\vec{\breve{u}}, \breve{p} )$ (which is geometry dependent). This means that
for regular (with respect to $\theta$) data, a small value for $N$ suffices.
\end{remark}

\appendix
\section{Comparison with Costabel et al.}
\label{app:sec:comparison}
In~\cite{2020arXiv200407216C}, characterization of $H^m(\breve{\Omega})$ by Fourier coefficients
(for any positive integer $m$) is treated first, then the relation~\cite[Proposition 6.1]{2020arXiv200407216C}
\begin{equation}
\label{eq:relation_vector_scalar_norm}
\| \vec{v}^k \|^2_{\boldsymbol{H}_{(k)}^1(\Omega)} =
\frac{1}{2} \, \| v_r^k + i v_{\theta}^k \|^2_{H_{(k+1)}^1(\Omega)} +
\frac{1}{2} \, \| v_r^k - i v_{\theta}^k \|^2_{H_{(k-1)}^1(\Omega)} +
\| v_z^k \|^2_{H_{(k)}^1(\Omega)},
\end{equation}
linking vector $\boldsymbol{H}_{(k)}^1(\Omega)$-norms and scalar $H_{(k)}^1(\Omega)$-norms, is used
to derive characterizations of $\bigl(H^m(\breve{\Omega})\bigr)^3$.

In the present work, we have derived the characterization of $\bigl(H^1(\breve{\Omega})\bigr)^3$
in Theorem~\ref{theorem:characth1} by directly rewriting the $\bigl(H^1(\breve{\Omega})\bigr)^3$-norm,
where the decomposition~\eqref{eq:H1vecnormcylcoord2Dsum} resulted in the
$\boldsymbol{H}_{(k)}^1(\Omega)$-norms
given by~\eqref{eq:H1vec0_norm} for $k=0$,
by~\eqref{eq:H1vec1_norm} for $|k|=1$,
and by~\eqref{eq:H1vec2_norm} for $|k|\geq 2$.

To compare the methods, we first note that
\begin{equation}
\label{eq:scalar_H1k_norm}
\| v \|_{H_{(k)}^1(\Omega)}^2 := \left\{
\begin{array}{lcl}
  \| v \|_{L^2_1(\Omega)}^2 + | v |_{H^1_1(\Omega)}^2, & \text{if} & k=0, \\ [2ex]
  \| v \|_{L^2_1(\Omega)}^2 + | v |_{H^1_1(\Omega)}^2 + k^2 \, \| v \|_{L^2_{-1}(\Omega)}^2, & \text{if} & |k| \geq 1,
\end{array}
\right.
\end{equation}
which follows, as in Section~\ref{subsec:Fourier_decomposition_of_norms},
by expressing the $H^1(\breve{\Omega})$-inner product in
cylindrical coordinates and using Fourier expansions:
\begin{align*}
\begin{alignedat}{1}
( \breve{u}, \breve{v} )_{H^1(\breve{\Omega})}
& = \int_{\breve{\Omega}}
  \left(
    \breve{u} \bar{\breve{v}}
    + \partial_x \breve{u} \partial_x \bar{\breve{v}}
    + \partial_y \breve{u} \partial_y \bar{\breve{v}}
    + \partial_z \breve{u} \partial_z \bar{\breve{v}}
  \right) \, \text{d}x \, \text{d}y \, \text{d}z \\
& = \int_{\tilde{\Omega}}
  \Bigl(
    u \bar{v}
     + \partial_r u \partial_r \bar{v}
     + \frac{1}{r^2} \, \partial_{\theta} u \partial_{\theta} \bar{v}
     + \partial_z u \partial_z \bar{v}
  \Bigr) \, r \, \text{d}r \, \text{d}\theta \, \text{d}z \\
  & = \sum_{k \in \mathbb{Z}} \int_{\Omega}
\Bigl(
  u^k \bar{v}^k
  + \partial_r u^k \partial_r \bar{v}^k
  + \partial_z u^k \partial_z \bar{v}^k + \frac{k^2}{r^2} \, u^k \bar{v}^k
\Bigr) \, r \, \text{d}r \, \text{d}z \\
 & =: \sum_{k \in \mathbb{Z}} (u^k, v^k)_{H_{(k)}^1(\Omega)}.
\end{alignedat}
\end{align*}

From~\eqref{eq:relation_vector_scalar_norm},~\eqref{eq:scalar_H1k_norm}, and the polarization identity
\begin{align*}
\| a + b \|^2 + \| a - b \|^2 = 2 \left( \| a \|^2 + \| b \|^2 \right),
\end{align*}
we then get for $k=0$
\begin{align*}
\| \vec{v}^0 \|^2_{\boldsymbol{H}_{(0)}^1(\Omega)} & =
\frac{1}{2} \, \| v_r^0 + i v_{\theta}^0 \|^2_{H_{(1)}^1(\Omega)} +
\frac{1}{2} \, \| v_r^0 - i v_{\theta}^0 \|^2_{H_{(-1)}^1(\Omega)} +
\| v_z^0 \|^2_{H_{(0)}^1(\Omega)} \\
& = \frac{1}{2} \left( \| v_r^0 + i v_{\theta}^0 \|_{L^2_1(\Omega)}^2
    + | v_r^0 + i v_{\theta}^0 |_{H^1_1(\Omega)}^2 + \| v_r^0
    + i v_{\theta}^0 \|_{L^2_{-1}(\Omega)}^2 \right. \\
& \qquad + \left. \| v_r^0 - i v_{\theta}^0 \|_{L^2_1(\Omega)}^2
    + | v_r^0 - i v_{\theta}^0 |_{H^1_1(\Omega)}^2
    + \| v_r^0 - i v_{\theta}^0 \|_{L^2_{-1}(\Omega)}^2 \right) \\
& \qquad + \| v_z^0 \|_{L^2_1(\Omega)}^2 + | v_z^0 |_{H^1_1(\Omega)}^2 \\
& = \| v_r^0 \|_{L^2_1(\Omega)}^2 + \| v_{\theta}^0 \|_{L^2_1(\Omega)}^2 + \| v_z^0 \|_{L^2_1(\Omega)}^2
    + | v_r^0 |_{H^1_1(\Omega)}^2 + | v_{\theta}^0 |_{H^1_1(\Omega)}^2 \\
& \qquad + \| v_r^0 \|_{L^2_{-1}(\Omega)}^2 + \| v_{\theta}^0 \|_{L^2_{-1}(\Omega)}^2
         + | v_z^0 |_{H^1_1(\Omega)}^2 \\
& = \| \vec{v}^0 \|_{(L^2_1(\Omega))^3}^2
    + \| v_r^0 \|_{V^1_1(\Omega)}^2 + \| v_{\theta}^0 \|_{V^1_1(\Omega)}^2 + | v_z^0 |_{H^1_1(\Omega)}^2,
\end{align*}
which is the same as~\eqref{eq:H1vec0_norm}.

Similarly, for $|k|=1$, we get
\begin{align*}
\| \vec{v}^{\pm 1} \|^2_{\boldsymbol{H}_{(\pm 1)}^1(\Omega)} & =
\frac{1}{2} \, \| v_r^{\pm 1} \pm i v_{\theta}^{\pm 1} \|^2_{H_{(\pm 2)}^1(\Omega)} +
\frac{1}{2} \, \| v_r^{\pm 1} \mp i v_{\theta}^{\pm 1} \|^2_{H_{(0)}^1(\Omega)} +
\| v_z^{\pm 1} \|^2_{H_{(\pm 1)}^1(\Omega)} \\
& = \frac{1}{2} \left( \| v_r^{\pm 1} \pm i v_{\theta}^{\pm 1} \|_{L^2_1(\Omega)}^2
                     + \| v_r^{\pm 1} \mp i v_{\theta}^{\pm 1} \|_{L^2_1(\Omega)}^2 \right. \\
& \qquad + \left. | v_r^{\pm 1} \pm i v_{\theta}^{\pm 1} |_{H^1_1(\Omega)}^2
                + | v_r^{\pm 1} \mp i v_{\theta}^{\pm 1} |_{H^1_1(\Omega)}^2 \right) \\
& \qquad + 2 \, \| v_r^{\pm 1} \pm i v_{\theta}^{\pm 1} \|_{L^2_{-1}(\Omega)}^2 \\
& \qquad + \| v_z^{\pm 1} \|_{L^2_1(\Omega)}^2
         + | v_z^{\pm 1} |_{H^1_1(\Omega)}^2
         + \| v_z^{\pm 1} \|_{L^2_{-1}(\Omega)}^2 \\
& = \| v_r^{\pm 1} \|_{L^2_1(\Omega)}^2
         + \| v_{\theta}^{\pm 1} \|_{L^2_1(\Omega)}^2
         + | v_r^{\pm 1} |_{H^1_1(\Omega)}^2
         + | v_{\theta}^{\pm 1} |_{H^1_1(\Omega)}^2 \\
& \qquad + 2 \, \| v_r^{\pm 1} \pm i v_{\theta}^{\pm 1} \|_{L^2_{-1}(\Omega)}^2
         + \| v_z^{\pm 1} \|_{L^2_1(\Omega)}^2
         + \| v_z^{\pm 1} \|_{V^1_{1}(\Omega)}^2 \\
& = \| \vec{v}^{\pm 1} \|_{(L^2_1(\Omega))^3}^2
         + 2 \, \| v_r^{\pm 1} \pm i v_{\theta}^{\pm 1} \|_{L^2_{-1}(\Omega)}^2 \\
& \qquad + | v_r^{\pm 1} |_{H^1_1(\Omega)}^2
         + | v_{\theta}^{\pm 1} |_{H^1_1(\Omega)}^2
         + \| v_z^{\pm 1} \|_{V^1_{1}(\Omega)}^2,
\end{align*}
which is the same as~\eqref{eq:H1vec1_norm}.

Finally, for $|k| \geq 2$ (noting that $|k \pm 1| \geq 1$), we get
\begin{align*}
\| \vec{v}^k \|^2_{\boldsymbol{H}_{(k)}^1(\Omega)} & =
\frac{1}{2} \, \| v_r^k + i v_{\theta}^k \|^2_{H_{(k+1)}^1(\Omega)} +
\frac{1}{2} \, \| v_r^k - i v_{\theta}^k \|^2_{H_{(k-1)}^1(\Omega)} +
\| v_z^k \|^2_{H_{(k)}^1(\Omega)} \\
& = \frac{1}{2} \left( \| v_r^k + i v_{\theta}^k \|_{L^2_1(\Omega)}^2
                     + \| v_r^k - i v_{\theta}^k \|_{L^2_1(\Omega)}^2 \right. \\
& \qquad +        | v_r^k + i v_{\theta}^k |_{H^1_1(\Omega)}^2
                + | v_r^k - i v_{\theta}^k |_{H^1_1(\Omega)}^2 \\
& \qquad + \left. (k+1)^2 \, \| v_r^k + i v_{\theta}^k \|_{L^2_{-1}(\Omega)}^2
         +        (k-1)^2 \, \| v_r^k - i v_{\theta}^k \|_{L^2_{-1}(\Omega)}^2 \right) \\
& \qquad + \| v_z^k \|_{L^2_1(\Omega)}^2
         + | v_z^k |_{H^1_1(\Omega)}^2
         + k^2 \, \| v_z^k \|_{L^2_{-1}(\Omega)}^2 \\
& = \| v_r^k \|_{L^2_1(\Omega)}^2 + \| v_{\theta}^k \|_{L^2_1(\Omega)}^2 + \| v_z^k \|_{L^2_1(\Omega)}^2 \\
& \qquad + | v_r^k |_{H^1_1(\Omega)}^2 + | v_{\theta}^k |_{H^1_1(\Omega)}^2 + | v_z^k |_{H^1_1(\Omega)}^2 \\
& \qquad + (1 + k^2) \left( \| v_r^k \|_{L^2_{-1}(\Omega)}^2 + \| v_{\theta}^k \|_{L^2_{-1}(\Omega)}^2 \right)
         + k^2 \, \| v_z^k \|_{L^2_{-1}(\Omega)}^2 \\
& \qquad + k \left( \| v_r^k + i v_{\theta}^k \|_{L^2_{-1}(\Omega)}^2 - \| v_r^k - i v_{\theta}^k \|_{L^2_{-1}(\Omega)}^2 \right) \\
& = \| \vec{v}^k \|_{(L^2_1(\Omega))^3}^2 + | \vec{v}^k |_{(H^1_1(\Omega))^3}^2 \\
& \qquad + (1 + k^2) \left( \| v_r^k \|_{L^2_{-1}(\Omega)}^2 + \| v_{\theta}^k \|_{L^2_{-1}(\Omega)}^2 \right)
         + k^2 \, \| v_z^k \|_{L^2_{-1}(\Omega)}^2 \\
& \qquad + 2ik \int_{\Omega} \bigl( v_{\theta}^k \bar{v}_r^k - v_r^k \bar{v}_{\theta}^k \bigr) \, \frac{1}{r} \, \text{d}r \, \text{d}z,
\end{align*}
which is the same as~\eqref{eq:H1vec2_norm}.

In conclusion, the $\boldsymbol{H}_{(k)}^1(\Omega)$-norms
for the Fourier coefficient spaces that we have derived
by directly rewriting the $\bigl(H^1(\breve{\Omega})\bigr)^3$-norm, are the
same as those obtained from the relation~\eqref{eq:relation_vector_scalar_norm}.
In addition to conveying structural understanding,
relation~\eqref{eq:relation_vector_scalar_norm} also facilitates the general
study, by induction, for positive integer order Sobolev spaces carried out in~\cite{2020arXiv200407216C}.
For our exposition, directly aimed at the Stokes problem, the method we have chosen
is straightforward to follow. The same remark can be made as to the difference
between introducing the differential operators
$\partial_{\zeta} = \frac{1}{\sqrt{2}} \, ( \partial_x - i \partial_y )$ and
$\partial_{\bar{\zeta}} = \frac{1}{\sqrt{2}} \, ( \partial_x + i \partial_y )$
in~\cite{2020arXiv200407216C}, and our choice of working directly
with $\partial_x$ and $\partial_y$.

\bibliography{niklas_referenser}
\bibliographystyle{plain}

\end{document}